\documentclass[reqno]{amsart}
\usepackage{amssymb}
\usepackage[usenames, dvipsnames]{color}
\usepackage{hyperref}
\usepackage{verbatim}

\theoremstyle{plain}
\newtheorem{theorem}{Theorem}[section]
\newtheorem{lemma}[theorem]{Lemma}

\newtheorem{proposition}[theorem]{Proposition}
\newtheorem{conjecture}[theorem]{Conjecture}

\theoremstyle{definition}

\theoremstyle{remark}
\newtheorem{remark}[theorem]{Remark}

\newcommand{\ef}{ \hfill $ \Box $ \vskip 3mm}
\newcommand{\bR}{{\mathbb R}}
\newcommand{\bN}{{\mathbb N}}

\def\la{\lambda}

\def\t{\tilde}

\def\q{\quad}
\def\qq{\qquad}
\def\th{\theta}
\def\g{\gamma}

\def\dl{\delta}

\def\lt{\left}
\def\les{\lesssim}
\def\rt{\right}

\def\i{\infty}
\def\e{\epsilon}

\def \ls{\lesssim}
\def\p{\partial}
\def\f{\frac}

\def\al{\alpha}

\def\s{\sqrt}

\def\be{\begin{equation}}
\def\ee{\end{equation}}
\def\bes{\begin{equation*}}
\def\ees{\end{equation*}}
\def\bali{\begin{aligned}}
\def\eali{\end{aligned}}

\numberwithin{equation}{section}

\makeatletter
\def\dashint{\operatorname%
{\,\,\text{\bf--}\kern-.98em\DOTSI\intop\ilimits@\!\!}}
\makeatother

\begin{document}

\title[Euler equations with vacuum and time-dependent damping]{Global existence and convergence to the modified Barenblatt solution for the compressible Euler equations with physical vacuum and time-dependent damping}

\author[X. Pan]{Xinghong Pan}
\address[X. Pan]{Department of Mathematics, Nanjing University of Aeronautics and Astronautics, Nanjing 211106, China}

\email{xinghong\_87@nuaa.edu.cn}

\thanks{X. Pan is supported by Natural Science Foundation of Jiangsu Province (No. SBK2018041027) and National Natural Science Foundation of China (No. 11801268).}



\begin{abstract}
 In this paper, the smooth solution of the physical vacuum problem for the one dimensional compressible Euler equations with time-dependent damping is considered. Near the vacuum boundary, the sound speed is $C^{1/2}$-H\"{o}lder continuous. The coefficient of the damping depends on time, given by this form $\frac{\mu}{(1+t)^\la}$, $\la,\ \mu>0$, which decays by order $-\la$ in time. Under the assumption that $0<\la<1,\ 0<\mu$ or $\la=1,\ 2<\mu$, we will prove the global existence of smooth solutions and convergence to the modified Barenblatt solution of the related porous media equation with time-dependent dissipation and the same total mass when the initial data of the Euler equations is a small perturbation of that of the Barenblatt solution. The pointwise convergence rates of the density, velocity and the expanding rate of the physical vacuum  boundary are also given. The proof is based on space-time weighted energy estimates, elliptic estimates and Hardy inequality in the Lagrangian coordinates. Our result is an extension of that in Luo-Zeng [Comm. Pure Appl. Math. 69 (2016), no. 7, 1354-1396], where the authors considered the physical vacuum free boundary problem of the compressible Euler equations with constant-coefficient damping.
\end{abstract}
\maketitle

\section{Introduction}
In this paper, we investigate the global existence of smooth solutions for the physical vacuum boundary problem of the following 1-d compressible Euler equations with time-dependent damping.
\begin{equation} \label{eeed}
\left\{
\begin{aligned}
&\rho_t+(\rho u)_x=0\q \text{in}\ I(t):=\{(x,t)|x_-(t)<x<x_+(t),t>0\}, \\
&(\rho u)_t+(p(\rho)+\rho u^2)_x=-\f{\mu}{(1+t)^\la}\rho u \q \text{in}\ I(t) ,\\
&\rho>0 \q\text{in}\ I(t),\q \rho=0 \q \text{on}\ x_\pm(t) ,\\
&(\rho,u)=(\rho_0,u_0) \q \text{on}\ I(0):=\{x|x_-(0)<x<x_+(0)\},
\end{aligned}
\rt.
\end{equation}
where the boundary $x_\pm(t)$ satisfies
\bes
\dot{x}_\pm(t)=u(x_\pm(t),t).
\ees

Here $(x, t) \in \mathbb{R} \times[0, \infty),\ \rho,\ u,$ and $p$ denote the space and time variable, density, velocity, and pressure, respectively. $\mathrm{I}(t),\ x_\pm(t),$ and $\dot{x}_\pm(t)$ represent the changing domain occupied by the gas, the moving vacuum boundary and the velocity of $x_b(t)$, respectively. $-\f{\mu}{(1+t)^\la}\rho u$, appearing on the right-hand side of $\eqref{eeed}_2$ describes the frictional damping which will decay by order $-\la$ in time. We assume the gas is the isentropic flow and the pressure satisfies the $\gamma$ law:
\bes
p(\rho)=\f{1}{\g}\rho^{\gamma} \quad \text { for } \gamma>1.
\ees
(Here the adiabatic constant is set to be $\f{1}{\g}$.) Let $c=\sqrt{p^{\prime}(\rho)}$ be the sound speed. A vacuum boundary is called physical if
\bes
0<\left|\frac{\partial c^{2}}{\partial x}\right|<+\infty
\ees
in a small neighborhood of the boundary. In order to capture this physical singularity, the initial density is supposed to satisfy
\bes
\rho_{0}(x)>0 \quad \text { for } x_{-}(0)<x<+x_{+}(0),
\ees
\bes
\rho_{0}\left(x_\pm(0)\right)=0, \quad \text { and } \quad 0<\left|\left(\rho_{0}^{\gamma-1}\right)_{x}\left(x_\pm(0)\right)\right|<\infty.
\ees

For the Euler equations with time-dependent damping, now there are numerous works concerning about the global existence, finite-time blow up, and asymptotic behaviors of smooth solutions. As far as the author knows, the pioneer works came from Hou-Witt-Yin \cite{HWY:2017NON, HWY:2018PJM} considering the multi-dimensional case and Pan \cite{Pan:2016NA, Pan:2016JMAA, Pan:2020AA} considering the one-dimensional case. They studied the Euler equations with damping term like $-\f{\mu}{(1+t)^\la}\rho u$ with $\la,\ \mu>0$. A critical couple of numbers $(\la,\mu)$, depending on the space dimension, are given to separate the global existence and finite-time blow up of smooth solutions when the initial data is a small perturbation of the equilibrium $(\rho,u)=(1,0)$. In particular, Pan \cite{Pan:2016NA, Pan:2016JMAA} proved that $(\la,\mu)=(1,2)$ are the critical couple numbers for the one-dimensional Euler equations which means when $0<\la<1,0<\mu$ or $\la=1,\ 2<\mu$, the global smooth solution exists, while when $\la=1,\ 0<\mu\leq2$ or $\la>1,\ \mu>0$, the smooth solution will blow up in finite time. Later, various results are shown in this aspect.  Sugiyama \cite{Sy:2018NA} studied the blow up mechanism of smooth solutions with $\la=1, \mu<2$ or $\la>1,\ \mu>0$. Li $et\ al.$ \cite{LLMZ:2017JMAA} and Cui $et\ al.$ \cite{CYZZ:2018JDE} proved the time asymptotic profile of solutions when $\la<1$ and $(\rho,u)$ approach to different constants at space infinity $+\i$ and $-\i$.  See also some recent works in \cite{CLLMZ:2020JDE, GLM:2020SIAM, JM:2020ARXIV1,JM:2020ARXIV2} and references therein.

Let $M \in(0, \infty)$ be the initial total mass, then the conservation law of mass, $\eqref{eeed}_1$, gives
\[
\int_{x_{-}(t)}^{x_{+}(t)} \rho(x, t) d x=\int_{x_{-}(0)}^{x+(0)} \rho_{0}(x) d x=: M \quad \text { for } t>0.
\]

The compressible Euler equations of isentropic flow with constant-coefficient damping ($\la=0,\ \mu=1$) are closely related to the porous media equations:
\be\label{epm1r}
\lt\{
\bali
&\rho_{t}=p(\rho)_{x x}\\
&p(\rho)_{x}=-\rho u.
\eali\rt.
\ee
For the solution of $\eqref{epm1r}_1$, basic understanding of the solution with finite mass is provided by Barenblatt (cf. \cite{Bgi:1953AMM}), with the following form

\be\label{epm2r}
\bar{\rho}(x, t)=(1+t)^{-\frac{1}{\gamma+1}}\left[A-B(1+t)^{-\frac{2}{\gamma+1}} x^{2}\right]^{\frac{1}{\gamma-1}}
\ee
where
\bes
B=\frac{\g-1}{2 (\gamma+1)}\q \text{ and}\q A^{\frac{\gamma+1}{2(\gamma-1)}}=M \sqrt{B}\left(\int_{-1}^{1}\left(1-y^{2}\right)^{1 /(\gamma-1)} d y\right)^{-1}.
\ees
And the velocity is given by
\be\label{epm2rr}
\bar{u}(x,t)=\f{x}{(\g+1)(1+t)}.
\ee

In \cite{LZ:2016CPAM}, the authors proved the global existence of smooth solutions and convergence to \eqref{epm2r} and \eqref{epm2rr} for the physical vacuum free boundary problem of the compressible Euler equations with constant-coefficient damping.

In this paper, we consider the time-dependent damping $-\f{\mu}{(1+t)^\la}\rho u$ with $0<\la<1,\ 0<\mu$ or $\la=1,2<\mu$, which decays as time goes to infinity.  We will prove the global existence of smooth solutions and convergence to the modified Barenblatt solution of the related porous media equation with time-dependent dissipation and the same total mass when the initial data of the Euler equations is a small perturbation of that of the modified Barenblatt solution.

For the time-dependent damping case, the related porous media equations with time-dependent dissipation read as follows
\be\label{epm}
\lt\{
\bali
\rho_{t}=\f{(1+t)^\la}{\mu}p(\rho)_{x x},\\
p(\rho)_{x}=-\f{\mu}{(1+t)^\la}\rho u.
\eali
\rt.
\ee
The related solution of $\eqref{epm}_1$ with finite mass is given by
\be\label{epm1}
\bar{\rho}(x, t)=(1+t)^{-\frac{1+\la}{\gamma+1}}\left[A-B(1+t)^{-\frac{2(1+\la)}{\gamma+1}} x^{2}\right]^{\frac{1}{\gamma-1}}
\ee
where

\bes
B=\frac{\mu(1+\la)(\gamma-1)}{2 (\gamma+1)}\q \text{ and}\q A^{\frac{\gamma+1}{2(\gamma-1)}}=M \sqrt{B}\left(\int_{-1}^{1}\left(1-y^{2}\right)^{1 /(\gamma-1)} d y\right)^{-1}.
\ees

We call this solution \eqref{epm1} the modified Barenblatt solution. Here the constant $A$ is chosen such that it has the same total mass as that for the solution of \eqref{eeed}:
\bes
\int_{\bar{x}_{-}(t)}^{\bar{x}_+(t)} \bar{\rho}(x, t) d x=M=\int_{x_{-}(t)}^{x_+(t)}\q \rho(x, t) d x \text { for }\ t \geq 0,
\ees
where
\bes
\bar{x}_{\pm}(t)=\pm \s{AB^{-1}}(1+t)^{\f{1+\la}{\g+1}}.
\ees
The corresponding velocity is defined by
\be\label{ebvelocity}
\bar{u}(x,t)=-\f{(1+t)^\la}{\mu}\f{p(\bar{\rho})_x}{\bar{\rho}}= \f{(1+\la)x}{(\g+1)(1+t)}.
\ee

%

  We will show the global existence of smooth solutions and convergence of $(\rho,u)$ to \eqref{epm1} and \eqref{ebvelocity} when the initial data of system \eqref{eeed} is a small perturbation of that for \eqref{epm1} and \eqref{ebvelocity}. In particular, the pointwise convergence rates of the density, the velocity and the expanding rate of the vacuum boundary in time are obtained.

The physical vacuum problem of the compressible Euler equations in which the sound speed is $C^{1 / 2}$ -H\"{o}lder continuous across the vacuum boundary is a challenging and interesting problem in the study of free boundary problems for compressible fluids. Even the local-in-time existence theory is hard to prove since standard methods of symmetric hyperbolic systems do not apply.

The phenomena of a physical vacuum arises naturally in several important physical situations such as the equilibrium and dynamics of boundaries of gaseous stars (cf \cite{Jj:2014CPAM,LXZ:2014ARMA}). The local-in-time well-posedness for the one and three dimensional compressible Euler equations with physical vacuum has been achieved by Coutand $et\ al.$ \cite{CLS:2010CMP,CS:2011CPAM,CS:2012ARMA} and Jang-Masmoudi \cite{JM:2009CPAM,JM:2015CPAM}. However, due to the strong degeneracy and singular behaviors near the vacuum boundary, it is a great challenge to extend the local-in-time existence theory to the global one of smooth solutions. In analyses, it is hard to establish the uniform-in-time higher-order a prior energy estimates to obtain the global-in-time regularity of solutions near vacuum boundaries. Huang-Marcati-Pan \cite{HMP:2005ARMA} and Huang-Pan-Wang \cite{HPW:2011ARMA} proved the $L^{p}$ convergence of $L^{\infty}$-weak solutions for the Cauchy problem of the one-dimensional compressible Euler equations with constant-coefficient damping to Barenblatt solutions of the porous media equations. They used entropy-type estimates for the solution itself without deriving estimates for derivatives. However, the interfaces separating gases and vacuum cannot be traced in the framework of $L^{\infty}$ -weak solutions.  In order to understand the behavior and long-time dynamics of physical vacuum boundaries, study on the global-in-time regularity of solutions is essential. To the best of our knowledge, the first global-in-time result of smooth solutions in Euler equations with constant coefficient damping comes from Luo-Zeng \cite{LZ:2016CPAM}, where the authors proved the global existence of smooth solutions and convergence to Barenblatt solutions for the physical vacuum free boundary problem. This result is somewhat surprising due to the difficulties mentioned above. In order to overcome difficulties in obtaining global-in-time regularities of solutions near vacuum boundaries, the authors in \cite{LZ:2016CPAM} constructed higher-order space and time weighted energy and performed higher-order nonlinear energy estimates and elliptic estimates. In the construction of higher-order weighted energy, the space weights are used to capture the behavior of solutions near vacuum states and the time weights detect the decay of solutions to Barenblatt solutions, respectively.

The a prior estimates for the weighted energy in the paper of Luo-Zeng \cite{LZ:2016CPAM} can be closed globally in time relies heavily on the constant-coefficient damping term $-\rho u$. When the damping vanishes, shock will form. For the mathematical analysis of finite-time formation of singularities, readers can see Alinhac \cite{As:1995BB}, Chemin \cite{Cjy:1990CMP}, Courant-Friedrichs \cite{CF:1948NY}, Christodoulou \cite{Cd:2007EMS}, Rammaha \cite{Rma:1989PROCEEDING} as well as Sideris \cite{St:1985CMP} and references therein for more detail.

It is natural to ask whether there are some global-in-time results for the Euler equations with decayed damping and vacuum. So here we consider the Euler equations with time-dependent damping and vacuum boundary. The damping term takes this form $-\f{\mu}{(1+t)^\la}\rho u$, which decays by order $-\la$ in time as $t$ goes to infinity. We think this issue is more challenging since now we not only have degenerate vacuum boundary but also have degenerate damping.

 From \cite{CLS:2010CMP,CS:2011CPAM,CS:2012ARMA,JM:2009CPAM,JM:2015CPAM}, a powerful tool in the study of physical vacuum free boundary problems of the Euler equations is the weighted energy estimate. By introducing the spatial weight to overcome the singularity at the vacuum boundary, the authors there establish the local-in-time well-posedness theory. Yet weighted estimates only involving spatial weights seem to be limited to proving local existence results. Later, Luo-Zeng \cite{LZ:2016CPAM,Zhh:2017ARMA} introduce time weights to quantify the large-time behavior of solutions for the Euler equations with constant coefficient damping in one dimension and three dimensions with spherically symmetric data. The choice of time weights is suggested by looking at the linearized problem to get hints on how the solution decays.

Inspired by their space-time weighted higher energy, we can construct a similar weighted energy to study the global existence of smooth solutions and convergence of the Euler equations with time-dependent damping \eqref{eeed}. In the case of $0<\la<1,\ \mu>0$, our time weight for the space-time mixed derivatives of the solution is a little weaker than that in \cite{LZ:2016CPAM}. See $\eqref{eenergy1}_2$ below and $(2.16)_2$ in \cite{LZ:2016CPAM}, which we think is reasonable due to the degeneration of the time-dependent damping. However, there is no difference for the time weight in the case $\la=1,\ \mu>2$, which seems to be a little strange.
  Also in our linearized equation, when $\la=1$, we need $\mu>2$ to ensure the closure of our time weighted energy despite in the lower and higher derivative estimates. This seems to be essential to prove the global existence of system \eqref{eeed} since in our previous papers \cite{Pan:2016NA,Pan:2016JMAA}, we have showed that when $\la=1$, $\mu=2$ is the threshold to separate the global existence and finite-time blow up of smooth solutions to system \eqref{eeed} when the initial data is a small perturbation of equilibrium $(\rho,u)=(1,0)$.


The strategy of our proof will follow the line with that in \cite{LZ:2016CPAM,Zhh:2017ARMA}. First to simplify the energy estimates, we will use elliptic estimates to show that the weighted space-derivative energy can be controlled by the time-derivative energy. In this process, we need to use the Hardy inequality repeatedly. Then we perform the time-derivative energy estimates in $L^2$ norms by the a prior assumption. one of our novelty is to perform the energy estimate for the linearized equations with degenerate time-dependent damping. See Lemma \ref{lbasice} below. To close the energy, the weighted $L^\i$ norms of the solutions are needed which can be achieved by Sobolev embedding and Hardy inequality. The advantages of this approach can prove the global existence and large-time convergence of solutions with the detailed convergence rates simultaneously.

Before ending this introduction, we review some prior results on vacuum free boundary problems for the compressible Euler equations and related modes besides the results mentioned above. Liu-Yang \cite{LY:1997JDE} proved the local existence theory when the singularity near the vacuum is mild in the sense that $c^{\alpha}\ (0<\al\leq 1)$ ($c$ denote the sound speed) is smooth across the vacuum boundary for the one-dimensional Euler equations with damping. Their method is based on the theory of symmetric hyperbolic systems which is not applicable to physical vacuum boundary problems since only $c^{2}$, instead of $c^{\alpha}$ is required to be smooth across the gas-vacuum interface (further development of this type of theory can be found in \cite{XY:2005JDE}). A nice review of singular behavior of solutions near vacuum boundaries for compressible fluids can be found in \cite{Yt:2006JCAM}.  An instability theory of stationary solutions to the physical vacuum free boundary problem for the spherically symmetric compressible Euler-Poisson equations of gaseous stars for $6 / 5<\gamma<4 / 3$ was established in Jang \cite{Jj:2014CPAM}. the local-in-time well-posedness of the physical vacuum free boundary problem was investigated in \cite{GL:2012JDE, GL:2016JMPA} for the one and three dimensional Euler-Poisson equations. See also some recent development in \cite{Zhh:2019ARXIV,Zhh:2020ARXIV} and references therein.

 Throughout the rest of paper, $C$ will denote a positive constant that only depends on the parameters of the problem $\la,\ \mu,\ \gamma$ and $C_{a,b,c,...}$ denotes a positive constant depending on $a,\,b,\, c,\,...$ which may be different from line to line.  We will employ the notation $a \lesssim b$ to denote $a \leq C b$ and $a \thickapprox b$ to denote $C^{-1} b \leq a \leq C b$.

\section{Reformulation of the problem and main results}

\subsection{Fix the domain and Lagrangian variables}
 We make the initial interval of the porous media solution \eqref{epm1}, $\left(\bar{x}_{-}(0), \bar{x}_{+}(0)\right)$, as the reference interval and define a diffeomorphism
$$
\eta_{0}:\left(\bar{x}_{-}(0), \bar{x}_{+}(0)\right) \rightarrow\left({x}_{-}(0), {x}_{+}(0)\right)
$$
by
\bes
\int^{\eta_{0}(x)}_{x_-(0)} \rho_{0}(y) d y=\int^{x}_{\bar{x}_-(0)} \bar{\rho}_{0}(y) d y\q  \text{for}\ x \in\left(0, +\i\right)
\ees
where
$\bar{\rho}_{0}(y):=\bar{\rho}(y, 0)$ is the initial density of the solution \eqref{epm1}. Differentiating the above equality by $x$ indicates
\be\label{elaginitial}
\rho_{0}\left(\eta_{0}(x)\right) \eta_{0}^{\prime}(x)=\bar{\rho}_{0}(x) \quad \text { for } x \in\left(\bar{x}_{-}(0), \bar{x}_{+}(0)\right).
\ee
To simplify the presentation, set
$$
\mathcal{I}:=\left(\bar{x}_{-}(0), \bar{x}_{+}(0)\right)=\left(-\s{AB^{-1}}, \s{AB^{-1}}\right).
$$
To fix the boundary, we transform system \eqref{eeed} into Lagrangian variables. For
$x \in \mathcal{I},$ we define the Lagrangian variable $\eta(x, t)$ by
\bes
\lt\{
\bali
&\eta_{t}(x, t)=u(\eta(x, t), t) \quad\text{ for}\ t>0,\\
&\eta(x, 0)=\eta_{0}(x),
\eali\rt.
\ees
and set the Lagrangian density and velocity by
\bes
f(x, t)=\rho(\eta(x, t), t) \quad \text {and } \quad v(x, t)=u(\eta(x, t), t).
\ees
Then the Lagrangian version of system \eqref{eeed} can be written on the reference domain $\mathcal{I}$ as
\be\label{elagrangian}
\lt\{
\begin{array}{ll}
f_{t}+f v_{x} / \eta_{x}=0 & \text{in}\ \mathcal{I} \times(0, \infty),\\
f v_{t}+\f{1}{\g}\left(f^{\gamma}\right)_{x} / \eta_{x}=-\f{\mu}{(1+t)^\la}f v \quad&\text{in}\ \mathcal{I} \times(0, \infty),\\
f>0\ \text{in}\ \mathcal{I} \times(0, \i), &f=0\ \text{on}\ \partial \mathcal{I} \times(0, \i),\\
(f, v)=\left(\rho_{0}\left(\eta_{0}\right), u_{0}\left(\eta_{0}\right)\right) \quad& \text{on}\ \mathcal{I} \times\{t=0\}.
\end{array}
\rt.
\ee
The map $\eta(\cdot, t)$ defined above can be extended to $\overline{\mathcal{I}}=\left[-\s{AB^{-1}}, \s{AB^{-1}}\right].$  In the setting, the vacuum free boundaries for problem \eqref{eeed} are given by
\bes
x_{\pm}(t)=\eta\left(\bar{x}_\pm(0), t\right)=\eta(\pm\s{AB^{-1}}, t)\q  \text{ for } t \geq 0.
\ees
It follows from solving $\eqref{elagrangian}_{1}$ and using \eqref{elaginitial} that
\be\label{emass}
 f(x, t) \eta_{x}(x, t)=\rho_{0}\left(\eta_{0}(x)\right) \eta_{0}^{\prime}(x)=\bar{\rho}_{0}(x),\q x \in \mathcal{I}.
\ee
It should be noticed that we need $\eta_{x}(x, t)>0$ for $x \in \mathcal{I}$ and $t \geq 0$ to make the Lagrangian transformation sensible, which will be verified later. So, the initial density, $\bar{\rho}_{0},$ can be regarded as a parameter, and system
\eqref{elagrangian} can be rewritten as
\be\label{elagrangian1}
\lt\{
\begin{array}{ll}
\bar{\rho}_{0} \eta_{t t}+\f{\mu}{(1+t)^\la}\bar{\rho}_{0} \eta_{t}+\f{1}{\g}\left(\bar{\rho}_{0}^{\gamma} / \eta_{x}^{\gamma}\right)_{x}=0  &\text{in}\ \mathcal{I} \times(0, \infty),\\
\left(\eta, \eta_{t}\right)=\left(\eta_{0}, u_{0}\left(\eta_{0}\right)\right) & \text{on}\ \mathcal{I} \times\{t=0\}.
\end{array}
\rt.
\ee
\subsection{Ansatz}
Define the Lagrangian variable $\bar{\eta}(x, t)$ for the modified Barenblatt flow in $\overline{\mathcal{I}}$ by
\[
\bar{\eta}_{t}(x, t)=\bar{u}(\bar{\eta}(x, t), t)=\frac{(\la+1)\bar{\eta}(x, t)}{(\gamma+1)(1+t)} \quad \text { for } t>0 \text { and } \bar{\eta}(x, 0)=x
\]
so that
\be\label{ean1}
\bar{\eta}(x, t)=x(1+t)^{\f{\la+1}{\g+1}} \quad \text { for }(x, t) \in \overline{\mathcal{I}} \times[0, \infty),
\ee
and
\bes
\f{\mu}{(1+t)^\la}\bar{\rho}_{0} \bar{\eta}_{t}+\f{1}{\g}\left(\bar{\rho}_{0}^{\gamma} / \bar{\eta}_{x}^{\gamma}\right)_{x}=0 \quad \text { in } \mathcal{I} \times(0, \infty).
\ees
Since $\bar{\eta}$ does not solve $\eqref{elagrangian1}_{1}$ exactly, we introduce a correction $h(t),$ which is the solution of the following initial value problem of ordinary differential equations:
\be\label{ecorf}
\lt\{
\bali
&h_{t t}+\f{\mu}{(1+t)^\la}h_{t}-\f{\mu(\la+1)}{\g+1}\left(\bar{\eta}_{x}+h\right)^{-\gamma} +\bar{\eta}_{x t t}+\f{\mu}{(1+t)^\la}\bar{\eta}_{x t}=0,\\
&h|_{t=0}=h_{t}|_{t=0}=0.
\eali
\rt.
\ee
(Notice that $\bar{\eta}_{x}, \bar{\eta}_{x t},$ and $\bar{\eta}_{x t t}$ are independent of $x$.) The new ansatz is then given by
\be\label{ecorf1}
\tilde{\eta}(x, t):=\bar{\eta}(x, t)+x h(t).
\ee
so that
\be\label{ecorf2}
\bar{\rho}_{0} \tilde{\eta}_{t t}+\f{\mu}{(1+t)^\la}\bar{\rho}_{0} \tilde{\eta}_{t}+\f{1}{\g}\left(\bar{\rho}_{0}^{\gamma} / \tilde{\eta}_{x}^{\gamma}\right)_{x}=0 \quad \text { in } \mathcal{I} \times(0, \infty).
\ee
It should be noticed that $\tilde{\eta}_{x}$ is independent of $x$. We will prove in the Appendix that $h$ is a positive bounded function and $\tilde{\eta}$ behaves similarly to $\bar{\eta}$. That is, there exist positive constants $K$ and $c_k$ independent of time $t$ such that for all $t \geq 0$,

\noindent {\bf If $\boldsymbol{0<\la<1}$}:
\be\label{ecor3}
\bali
&(1+t)^{\f{\la+1}{\gamma+1}} \leq \tilde{\eta}_{x}(t) \leq K(1+t)^{\f{\la+1}{\gamma+1}}, \quad \tilde{\eta}_{x t}(t) \geq 0,\\
&\lt|\f{d^k\t{\eta}_x(t)}{dt^k}\rt|\leq c_k(1+t)^{\f{\la+1}{\g+1}-k}\q \text{for}\q k\in\bN.
\eali
\ee

\noindent {\bf If $\boldsymbol{\la=1}$}:

\be\label{ecor4}
\bali
&(1+t)^{\f{2}{\g+1}} \leq \tilde{\eta}_{x}(t) \leq K(1+t)^{\f{2}{\g+1}}, \quad \tilde{\eta}_{x t}(t) \geq 0,\\
&\lt|\f{d^k\t{\eta}_x(t)}{dt^k}\rt|\leq
\lt\{
\bali
&c_k(1+t)^{\f{2}{\g+1}-k}\q \text{for}\q k<\mu+\f{2}{\g+1}\ \text{and}\ k\in\bN,\\
&c_k(1+t)^{-\mu}\ln(1+t) \q \text{for}\q k\geq\mu+\f{2}{\g+1}\ \text{and}\ k\in\bN.
\eali
\rt.
\eali
\ee

\subsection{Main results.}

Let
\be  \label{eerro}
 w(x,t)=\eta(x,t)-\t{\eta}(x,t).
\ee Then subtracting \eqref{ecorf2} from $\eqref{elagrangian1}_1$, we see that $w$ satisfies
\be\label{eerr}
\lt\{
\bali
&\bar{\rho}_{0} w_{t t}+\f{\mu}{(1+t)^\la}\bar{\rho}_{0} w_{t}+\f{1}{\g}\left[\bar{\rho}_{0}^{\gamma} \lt((\t{\eta}_x+w_x)^{-\g}-\t{\eta}^{-\g}_x\rt)\right]_{x}=0\\
 &\qq\qq\qq\qq\qq\qq\qq\qq\q\ \   \text{in}\ \mathcal{I} \times(0, \infty),\\
&\left(w, w_{t}\right)=\left(\eta_{0}-x, u_{0}\left(\eta_{0}\right)-\f{\la+1}{\g+1}x\right) \q \text{on}\ \mathcal{I} \times\{t=0\}.
\eali
\rt.
\ee

In the rest of the paper, we will use the notation
$$
\int:= \int_{\mathcal{I}}, \quad\|\cdot\|:=\|\cdot\|_{L^{2}(\mathcal{I})}, \quad \text { and } \quad\|\cdot\|_{L^{\infty}}:=\|\cdot\|_{L^{\infty}(\mathcal{I})}.
$$
Denote $\al=\f{1}{\g-1}$ and set
\bes
m=\lt\{
\begin{array}{ll}
4+[\al]  &\text{if}\ \la<1,\\
\min\lt\{4+[\al], [\mu+2/(\g+1)]\rt\} & \text{if}\ \la=1.
\end{array}
\rt.
\ees
 Let $\dl\in(0,\f{2(\la+1)}{\g+1})$.  For $j=0, \ldots, m$ and $i=0, \ldots, m-j,$ we set

\bes
\begin{aligned}
&\mathcal{E}_{j}(t):=(1+t)^{2 j-\dl{\bf 1_{\la<1}}} \int_{\mathcal{I}} \lt[\bar{\rho}_{0}(\partial_{t}^{j} w)^{2}+\bar{\rho}_{0}^{\gamma}(\partial_{t}^{j} w_{x})^{2}\rt.\\
                             &\qq\qq\qq\left.+(1+t)^{\la+1} \bar{\rho}_{0}(\partial_{t}^{j+1} w)^{2}\right](x, t) d x,\\
&\mathcal{E}_{j, i}(t):=(1+t)^{2 j-\dl{\bf 1_{\la<1}}} \int_{\mathcal{I}} \left[\bar{\rho}_{0}^{1+(i-1)(\gamma-1)}(\partial_{t}^{j} \partial_{x}^{i} w)^{2}\rt.\\
  &\qq\qq\qq\qq\qq+ \lt.\bar{\rho}^{1+(i+1)(\g-1)}_{0}(\partial_{t}^{j+1} w)^{2}\right](x, t) d x,
\end{aligned}
\ees
where ${\bf 1}_{\la<1}$ is the characteristic function on $\{\la<1\}$, which means
\bes
{\boldsymbol 1}_{\la<1}=\lt\{
\bali
& 1,\qq \text{if}\ \la<1,\\
& 0,\qq \text{if}\ \la=1.
\eali
\rt.
\ees
If we set
$$
\sigma(x):=\bar{\rho}_{0}^{\gamma-1}(x)=A-Bx^2, \quad x \in \mathcal{I},
$$
then $\mathcal{E}_{j}$ and $\mathcal{E}_{j, i}$ can be rewritten as
\be\label{eenergy1}
\begin{aligned}
&\mathcal{E}_{j}(t)=(1+t)^{2 j-\dl{\bf 1}_{\la<1}} \int_{\mathcal{I}}\left[\sigma^{\alpha}\left(\partial_{t}^{j} w\right)^{2}+\sigma^{\alpha+1}\left(\partial_{t}^{j} w_{x}\right)^{2}\right.\\
&\qq\qq\qq \left.+(1+t)^{\la+1} \sigma^{\alpha}\left(\partial_{t}^{j+1} w\right)^{2}\right](x, t) d x,\\
&\mathcal{E}_{j, i}(t)=(1+t)^{2 j-\dl{\bf 1}_{\la<1}} \int_{\mathcal{I}}\left[\sigma^{\alpha+i+1}\left(\partial_{t}^{j} \partial_{x}^{i+1} w\right)^{2}\rt.\\
&\qq\qq\qq\qq\qq\lt.+\sigma^{\alpha+i-1}\left(\partial_{t}^{j} \partial_{x}^{i} w\right)^{2}\right](x, t) d x.
\end{aligned}
\ee

\begin{remark}
Our weighted energy defined in \eqref{eenergy1} is similar with but different from that in Luo-Zeng \cite{LZ:2016CPAM}, especially for the time weights.
\end{remark}
The total energy is defined by
$$
\mathcal{E}(t):=\sum_{j=0}^{m}\left(\mathcal{E}_{j}(t)+\sum_{i=1}^{m-j} \mathcal{E}_{j, i}(t)\right).
$$

 Now, we are ready to state the main result.

\begin{theorem} \label{thmain}
Suppose that $\la=1,\mu>2$ or $0<\la<1,\mu>0$. There exists a constant $\e_0$ such that if $\mathcal{E}(0) \leq \e_0,$ then the problem \eqref{eerr} admits a global unique smooth solution in $\mathcal{I} \times[0, \infty)$ satisfying for all $t \geq 0$
$$
\mathcal{E}(t) \leq C \mathcal{E}(0)
$$
and
\be\label{elinfinity}
 \bali
&\sup _{x \in \mathcal{I}}\left\{\sum_{j=0}^{3}(1+t)^{2 j-\dl{\bf 1}_{\la<1}}\left|\partial_{t}^{j} w(x, t)\right|^{2}+\sum_{j=0}^{1}(1+t)^{2 j-\dl{\bf 1}_{\la<1}}\left|\partial_{t}^{j} w_{x}(x, t)\right|^{2}\right\}\\
&\q+\sup _{x \in \mathcal{I}} \sum_{i+j \leq m, 2 i+j \geq 4}(1+t)^{2 j-\dl{\bf 1}_{\la<1}}\left|\sigma^{\frac{2 i+j-3}{2}} \partial_{t}^{j} \partial_{x}^{i} w(x, t)\right|^{2} \leq C \mathcal{E}(0),
\eali
\ee
where $C$ is a positive constant independent of $t$.
\end{theorem}

As a corollary of Theorem \ref{thmain}, we have the following theorem for solutions to the original vacuum free boundary problem \eqref{eeed}.

\begin{theorem}\label{thoriginal}
Suppose that $\la=1,\mu>2$ or $0<\la<1,\mu>0$. There exists a constant $\e_0>0$ such that if $\mathcal{E}(0) \leq \e_0,$ then the problem \eqref{eeed} admits a global unique smooth solution $(\rho, u, \mathrm{I}(t))$ for $t \in[0, \infty)$ satisfying
\be\label{eod}
|\rho(\eta(x, t), t)-\bar{\rho}(\bar{\eta}(x, t), t)| \leq C\left(A-B x^{2}\right)^{\frac{1}{\gamma-1}}(1+t)^{-\frac{2(\la+1)}{\gamma+1}+\f{\dl}{2}{\bf1}_{\la<1}},
\ee
\be\label{eov}
|u(\eta(x, t), t)-\bar{u}(\bar{\eta}(x, t), t)| \leq C(1+t)^{\f{\la-\g}{\g+1}},
\ee
\be\label{eob}
 x_{\pm}(t) \thickapprox \pm(1+t)^{\frac{\la+1}{\gamma+1}},
\ee
\be\label{eobd}
\left|\frac{d^{k} x_{\pm}(t)}{d t^{k}}\right| \leq C(1+t)^{\frac{\la+1}{\gamma+1}-k}, \quad k=1,2,3,
\ee
for all $x \in \mathcal{I}$ and $t \geq 0 .$ Here $C$ is a positive constant, depending on $\e_0$ and the upper bound of $h$ but independent of $t$.
\end{theorem}
The pointwise behavior of the density and the convergence of the velocity for the vacuum free boundary problem \eqref{eeed} to that of the modified Barenblatt solution are given by \eqref{eod} and \eqref{eov}, respectively. \eqref{eob} gives the precise expanding rate of the vacuum boundaries, which is the same as that for the modified Barenblatt solution. It is also shown in \eqref{eod} that the difference of density to problem \eqref{eeed} and the corresponding Barenblatt density decays at the rate of $(1+t)^{-[2(\la+1) /(\gamma+1)]^-}$ in $L^{\infty},$ where $a^-$ denotes a constant which is smaller than but close to $a$, while the density of the modified Barenblatt solution, $\bar{\rho},$ decays at the rate of $(1+t)^{-(\la+1) /(\gamma+1)}$ in $L^{\infty}$.

Due to the the finite-time blow up of smooth solutions for \eqref{eerr} with $\la=1,0<\mu\leq2$ or $\la>1,\mu>0$ in \cite{Pan:2016NA,Pan:2016JMAA} under the assumption that the density and the velocity is a small perturbation of $(\bar{\rho},\bar{u})=(1,0)$, we give the following conjecture, which will be considered in our further work.

\begin{conjecture}
Suppose that $\la=1,0<\mu\leq2$ or $\la>1,\mu>0$. The smooth solution of \eqref{eerr} will blow up in finite time for a family of smooth initial data $(w,\p_t w)|_{t=0}$  even if $(w,\p_t w)|_{t=0}$ is sufficiently small.
\end{conjecture}

\section{Proof of Theorem \ref{thmain}}

At the beginning, we give a weighted Sobolev $L^\i$ embedding Lemma for later use.
\begin{lemma} \label{laprioriinfty}
Suppose that $\mathcal{E}(t)$ is finite, then it holds that
\bes
\bali
&\sup _{x \in \mathcal{I}}\left\{\sum_{j=0}^{3}(1+t)^{2 j-\dl{\bf 1}_{\la<1}}\left|\partial_{t}^{j} w(x, t)\right|^{2}+\sum_{j=0}^{1}(1+t)^{2 j-\dl{\bf 1}_{\la<1}}\left|\partial_{t}^{j} w_{x}(x, t)\right|^{2}\right\}\\
&\q+\sup _{x \in \mathcal{I}} \sum_{i+j \leq m, 2 i+j \geq 4}(1+t)^{2 j-\dl{\bf 1}_{\la<1}}\left|\sigma^{\frac{2 i+j-3}{2}} \partial_{t}^{j} \partial_{x}^{i} w(x, t)\right|^{2} \leq C \mathcal{E}(t).
\eali
\ees
\end{lemma}

The proof of Lemma \ref{laprioriinfty} follows by the same line as that in Lemma 3.7 of \cite{LZ:2016CPAM}. Here we omit the details.

The proof of Theorem \ref{thmain} is based on the local existence of smooth solutions (cf. \cite{CS:2011CPAM,JM:2009CPAM}) and continuation arguments. The uniqueness of smooth solutions can be obtained as in section 11 of \cite{LXZ:2014ARMA}. In order to prove the global existence of smooth solutions, we need to obtain the uniform-in-time a priori estimates on any given time interval $[0, T]$ satisfying $\sup _{t \in[0, T]} \mathcal{E}(t)<\infty .$ To this end, we use a bootstrap argument by making the following a priori assumption: there exists a suitably small fixed positive number $\epsilon_{0} \in(0,1)$ independent of $t$ such that
\be\label{eenergyass}
\sup_{0\leq t\leq T}\mathcal{E}(t)\leq M\e_0
\ee
for some constant $M$, independent of $\e_0$, to be determined later. Under this a priori assumption,
and by using Lemma \ref{laprioriinfty}, we see that
\be\label{esobolev}
\bali
&\sup _{x \in \mathcal{I}}\left\{\sum_{j=0}^{3}(1+t)^{2 j-\dl{\bf 1}_{\la<1}}\left|\partial_{t}^{j} w(x, t)\right|^{2}+\sum_{j=0}^{1}(1+t)^{2 j-\dl{\bf 1}_{\la<1}}\left|\partial_{t}^{j} w_{x}(x, t)\right|^{2}\right\}\\
&\q+\sup _{x \in \mathcal{I}} \sum_{i+j \leq m, 2 i+j \geq 4}(1+t)^{2 j-\dl{\bf 1}_{\la<1}}\left|\sigma^{\frac{2 i+j-3}{2}} \partial_{t}^{j} \partial_{x}^{i} w(x, t)\right|^{2} \leq C M\e_0.
\eali
\ee
Here we can assume that $M\e_0$ is sufficiently small such that $M\e_0\ll 1$. Then we show in subsection \ref{subelliptic} the following elliptic estimates:
\be\label{elleptic}
\mathcal{E}_{j, i}(t) \leq C\sum_{\ell=1}^{i+j} \mathcal{E}_{\ell}(t) \quad \text { when } i,\ j \geq 0,\ i+j \leq m,
\ee
where $C$ is a positive constant independent of $t$.

With \eqref{esobolev} and elliptic estimates \eqref{elleptic}, we show in subsection \ref{subweight} the following nonlinear weighted energy estimate: for some positive constant $C$ independent of $t$
\be\label{eweight}
\mathcal{E}_{j}(t) \leq C \sum_{\ell=0}^{j} \mathcal{E}_{\ell}(0), \quad j=0,1, \ldots, m.
\ee
Combining \eqref{elleptic} and \eqref{eweight}, we see that
\be\label{eweight1}
\mathcal{E}(t) \leq  C_\ast \mathcal{E}(0),
\ee
for some constant $C_\ast$ independent of $t$ and $M$.  By choosing $M=2C_\ast$, we see that
\bes
\mathcal{E}(t) \leq  \f{1}{2}M \e_0,
\ees
which closes energy estimates.
\subsection{Preliminaries}
In this subsection, we present some embedding estimates for weighted Sobolev spaces that will be used later.

 Set
\[
\begin{array}{c}
d(x):=\operatorname{dist}(x, \partial \mathcal{I})=\min \{x+\sqrt{AB^{-1}}, \sqrt{AB^{-1}}-x\}, \\
x \in \mathcal{I}=(-\sqrt{AB^{-1}}, \sqrt{AB^{-1}}).
\end{array}
\]
For any $a>0$ and nonnegative integer $b,$ the weighted Sobolev space $H^{a, b}(\mathcal{I})$ is given by
\[
H^{a, b}(\mathcal{I}):=\left\{d^{a / 2} F \in L^{2}(\mathcal{I}): \int_{\mathcal{I}} d^{a}\left|\partial_{x}^{k} F\right|^{2} d x<\infty, 0 \leq k \leq b\right\}
\]
with the norm
\[
\|F\|_{H^{a, b}(\mathcal{I})}^{2}:=\sum_{k=0}^{b} \int_{\mathcal{I}} d^{a}\left|\partial_{x}^{k} F\right|^{2} d x.
\]
Then for $b \geq a / 2,$ we have the following embedding of weighted Sobolev spaces
(cf. \cite{KMP:2007} ):
\[
H^{a, b}(\mathcal{I}) \hookrightarrow H^{b-a / 2}(\mathcal{I})
\]
with the estimate
\bes
\|F\|_{H^{b-a / 2}(\mathcal{I})} \leq C_{a,b}\|F\|_{H^{a, b}(\mathcal{I})}
\ees
for some positive constant $C_{a,b}$. Obviously, $\sigma(x)$ is equivalent to $d(x)$, that is,
\bes
\s{AB} d(x)\leq \sigma(x)\leq 2\s{AB} d(x).
\ees
So, we have
 \bes
\|F\|_{H^{b-a / 2}(\mathcal{I})} \leq C_{a,b}\sum_{k=0}^{b} \int_{\mathcal{I}} \sigma^{a}\left|\partial_{x}^{k} F\right|^{2} d x.
\ees
 The following general version of the Hardy inequality, whose proof can be found in \cite{KMP:2007}, will also be used frequently in this paper. Let $\th>1$ be a given real number and $F$ be a function satisfying
\[
\int_{0}^{L} x^{\th}\left(F^{2}+F_{x}^{2}\right) d x<\infty,
\]
where $L$ is a positive constant; then it holds that
\[
\int_{0}^{L} x^{\th-2} F^{2} d x \leq C_{\th,L} \int_{0}^{L} x^{\th}\left(F^{2}+F_{x}^{2}\right) d x.
\]
As a consequence, by making a simple variable change, we can show
\be\label{ehardy}
\bali
\int \sigma^{\th-2}F^2dx\thickapprox&\int d^{\th-2} F^2dx\\
                      \leq &\int d^{\th}(F^2+F^2_x)dx\\
                      \thickapprox&\int \sigma^{\th}(F^2+F^2_x) dx.
\eali
\ee

\subsection{Elliptic Estimates}\label{subelliptic}

We prove the following elliptic estimates in this subsection.
\begin{proposition}\label{pelliptic}
Under the assumption of \eqref{eenergyass} for suitably small positive number $\epsilon_0\in (0,1)$, then for $0 \leq t \leq T$, we have
\bes
 \mathcal{E}_{j, i}(t) \lesssim \sum_{\ell=0}^{i+j} \mathcal{E}_{\ell}(t) \quad \text { when } i,\ j \geq 0,\ i+j \leq m.
\ees
\end{proposition}
The proof of this proposition consists of Lemma \ref{llowee} and Lemma \ref{lhighe} below.

\noindent\textbf{Lower-Order Elliptic Estimates}

 Equation $\eqref{eerr}_1$ can be rewritten as
\bes
\bali
\t{\eta}^{-\g-1}_x\left(\bar{\rho}_{0}^{\gamma} w_{x}\right)_{x}&=\bar{\rho}_{0} w_{t t}+\f{\mu}{(1+t)^\la}\bar{\rho}_{0} w_{t}\\
&+\f{1}{\g}\left[\bar{\rho}_{0}^{\gamma}\left(\left(\t{\eta}_x+w_{x}\right)^{-\gamma}-\t{\eta}^{-\g}_x+\gamma \t{\eta}^{-\g-1}_xw_{x}\right)\right]_{x}.
\eali
\ees
Divide the equation above by $\bar{\rho}_{0}$ and expand the resulting equation to obtain
\be\label{eelliptic}
\begin{aligned}
& \t{\eta}^{-\g-1}_x\lt[\sigma w_{x x}+(1+\al) \sigma_x w_{x}\rt] \\
=&w_{t t}+\f{\mu}{(1+t)^\la}w_{t}-\sigma\left[\left(\t{\eta}_x+w_{x}\right)^{-\gamma-1}-\t{\eta}^{-\g-1}_x\right] w_{x x}\\
&+\al\sigma_x \left[\left(\t{\eta}_x+w_{x}\right)^{-\gamma}-\t{\eta}^{-\g}_x+\gamma  \t{\eta}^{-\g-1}_x w_{x}\right].
\end{aligned}
\ee
\begin{lemma} \label{llowee}
Under the assumption of \eqref{eenergyass} for suitably small positive number $\epsilon_0\in (0,1)$. Then
\bes
\mathcal{E}_{0,0}(t)+\mathcal{E}_{1,0}(t)+\mathcal{E}_{0,1}(t) \lesssim \mathcal{E}_{0}(t)+\mathcal{E}_{1}(t),\ 0 \leq t \leq T.
\ees
\end{lemma}
\begin{proof}
When $i=0$, we using  \eqref{ehardy} to see that
\bes
\bali
\mathcal{E}_{j,0}(t)&:= (1+t)^{2j-\dl{\bf 1}_{\la<1}}\int\left[\sigma^{\alpha+1}(\partial_{t}^{j} w_x)^{2}+\sigma^{\alpha-1}(\partial_{t}^{j} w)^{2}\right](x, t) d x\\
            &\ls (1+t)^{2j-\dl{\bf 1}_{\la<1}}\int\sigma^{\alpha+1}\lt[(\partial_{t}^{j} w_x)^{2}+(\partial_{t}^{j} w)^{2}\rt]dx\\
            &\ls (1+t)^{2j-\dl{\bf 1}_{\la<1}}\int\left[\sigma^{\alpha+1}(\partial_{t}^{j} w_x)^{2}+\sigma^\al(\partial_{t}^{j} w)^{2}\rt]dx\\
            &\ls \mathcal{E}_{j}(t),
\eali
\ees
where at the last but second line, we have used the fact that $|\sigma|\leq A$. This implies that $\mathcal{E}_{0,0}(t)+\mathcal{E}_{1,0}(t) \lesssim \mathcal{E}_{0}(t)+\mathcal{E}_{1}(t)$. We mainly focus on the the proof of $\mathcal{E}_{0,1}(t) \lesssim \mathcal{E}_{0}(t)+\mathcal{E}_{1}(t).$

Remembering that $\t{\eta}_x\thickapprox (1+t)^{\f{\la+1}{\g+1}}$, multiply equation \eqref{eelliptic} by $\t{\eta}^{\g+1}_x\sigma^{\alpha / 2}$ and perform the spatial $L^2$-norm to obtain
\be\label{eelliptic1}
\bali
&\left\|\sigma^{1+\frac{\alpha}{2}} w_{x x}+(1+\al) \sigma^{\frac{\alpha}{2}}\sigma_x w_{x}\right\|^{2} \\
&\leq C \left((1+t)^{2(\la+1)}\left\|{\sigma}^{\frac{\alpha}{2}} w_{t t}\right\|^{2}+(1+t)^{2}\left\|{\sigma}^{\frac{\alpha}{2}} w_{t}\right\|^{2}\right)\\
&+C(1+t)^{2(\la+1)}\lt(\t{\eta}^{-2}_x\left\|{\sigma}^{1+\frac{\alpha}{2}} w_{x} w_{x x}\right\|^{2}+\t{\eta}^{-2}_x\left\|{\sigma}^{\frac{\alpha}{2}} \sigma_x w_{x}^{2}\right\|^{2}\rt)\\
&\leq C (1+t)^{\dl{\bf1}_{\la<1}}\mathcal{E}_{1}\\
&+C\left\|w_{x}\right\|_{L^{\infty}}^{2}\t{\eta}^{-2}_x\left(\left\|{\sigma}^{1+\frac{\alpha}{2}} w_{x x}\right\|^{2}+\left\|\sigma^{\frac{\alpha}{2}} \sigma_x w_{x}\right\|^{2}\right),
\eali
\ee
where we have used the Taylor expansion, the smallness of $\t{\eta}^{-1}_xw_{x}$ (which is the consequence of \eqref{esobolev}) to derive the first inequality and the definition of $\mathcal{E}_{1}$ to the second. Note that by integration by parts, the left-hand side of \eqref{eelliptic1} can be expanded as
\be\label{eelliptic2}
\bali
&\left\|\sigma^{1+\frac{\alpha}{2}} w_{x x}+(1+\al) \sigma^{\frac{\alpha}{2}}\sigma_x w_{x}\right\|^{2} \\
=&\left\|{\sigma}^{1+\frac{\alpha}{2}} w_{x x}\right\|^{2}+(1+\al)^{2}\left\|{\sigma^{\f{\al}{2}}}\sigma_x w_{x}\right\|^{2}+(1+\al) \int \sigma^{1+\alpha} \sigma_x\left(w_{x}^{2}\right)_{x} d x \\
=&\left\|{\sigma}^{1+\frac{\alpha}{2}} w_{x x}\right\|^{2}-(1+\al) \int \sigma^\alpha\sigma_{xx} w_{x}^{2} d x\\
=&\left\|{\sigma}^{1+\frac{\alpha}{2}} w_{x x}\right\|^{2}+2(1+\al)B \int \sigma^\alpha w_{x}^{2} d x\\
\thickapprox& (1+t)^{\dl{\bf1}_{\la<1}}\mathcal{E}_{0,1}.
\eali
\ee
At the fourth line of the above inequality, we have used the fact that $\sigma_{xx}=-2B$.

By combining \eqref{eelliptic1} and \eqref{eelliptic2}, we get
\be\label{eelliptic3}
\bali
&\mathcal{E}_{0,1}\leq C\mathcal{E}_{1}+C\e_0\t{\eta}^{-2}_x\left(\left\|{\sigma}^{1+\frac{\alpha}{2}} w_{x x}\right\|^{2}+\left\|\sigma^{\frac{\alpha}{2}} \sigma_x w_{x}\right\|^{2}\right),
\eali
\ee
where we have used \eqref{esobolev} to estimate that $\|w_{x}\|_{L^{\infty}}^{2}\ls \e_0(1+t)^{\dl{\bf1}_{\la<1}}$.
On the other hand, by using \eqref{eelliptic2}
\bes
\begin{aligned}
&\left\|(\al+1)\sigma^{\f{\al}{2}}\sigma_x w_{x}\right\|^{2} =\left\|\sigma^{1+\frac{\alpha}{2}} w_{x x}+(\al+1) \sigma_x\sigma^{\frac{\alpha}{2}} w_{x}-\sigma^{1+\frac{\alpha}{2}} w_{x x}\right\|^{2} \\
\leq& 2\left\|\sigma^{1+\frac{\alpha}{2}} w_{x x}+(\al+1)\sigma_x \sigma^{\frac{\alpha}{2}} w_{x}\right\|^{2}+2\left\|\sigma^{1+\frac{\alpha}{2}}  w_{x x}\right\|^{2} \\
\ls& (1+t)^{\dl{\bf1}_{\la<1}}\mathcal{E}_{0,1}.
\end{aligned}
\ees
This, together with \eqref{eelliptic3}, gives
\bes
\bali
\mathcal{E}_{0,1}\ls& \mathcal{E}_{1}+\e_0(1+t)^{\dl{\bf1}_{\la<1}}\t{\eta}^{-2}_x\mathcal{E}_{0,1}\\
                    \ls& \mathcal{E}_{1}+\e_0(1+t)^{(1+t)^{\dl{\bf1}_{\la<1}-\f{2(\la+1)}{\g+1}}}\mathcal{E}_{0,1}\\
                    \ls& \mathcal{E}_{1}+\e_0\mathcal{E}_{0,1},
\eali
\ees
where we have choose $\dl\in (0,\f{2(\la+1)}{\g+1})$.
This implies, with the aid of the smallness of $\e_0$, that
\bes
\mathcal{E}_{0,1} \leq C \mathcal{E}_{1}(t).
\ees
\end{proof}
\noindent {\bf Higher-Order Elliptic Estimates}

For $i\geq 1$ and $j\geq 0,$ applying $\p^j_t\p^{i-1}_x$ to \eqref{eelliptic} yields that
\be\label{ehelliptic}
\bali
&\t{\eta}^{-\g-1}_x\lt[\sigma\p^j_t\p^{i+1}_x w+\lt(\al+i\rt)\sigma_x\p^j_t\p^{i}_x w\rt]\\
=&\p^{j+2}_t\p^{i-1}_x w+\f{\mu}{(1+t)^\la}\p^{j+1}_t\p^{i-1}_x w + Q_1+Q_2+Q_3,
\eali
\ee

where
\be\label{ehelliptic5}
\begin{aligned}
&Q_{1}:=- \sum_{\ell=1}^{j}\left[\partial_{t}^{\ell}\left(\tilde{\eta}_{x}^{-\gamma-1}\right)\right] \partial_{t}^{j-\ell}\left[\sigma \partial_{x}^{i+1} w+(\alpha+i)\sigma_x \partial_{x}^{i} w\right]\\
&\quad- \partial_{t}^{j}\left\{\tilde{\eta}_{x}^{-\gamma-1}\left[\sum_{\ell=2}^{i-1} C_{i-1}^{\ell}\left(\partial_{x}^{\ell} \sigma\right)\left(\partial_{x}^{i+1-\ell} w\right)\rt.\rt.\\
&\qq\qq\lt.\lt.+(\alpha+1) \sum_{\ell=1}^{i-1} C_{i-1}^{\ell}\left(\partial_{x}^{\ell+1} \sigma\right)\left(\partial_{x}^{i-\ell} w\right)\right]\right\}, \\
&Q_{2}:=- \partial_{t}^{j} \partial_{x}^{i-1}\left\{\sigma\left[\left(\tilde{\eta}_{x}+w_{x}\right)^{-\gamma-1}-\tilde{\eta}_{x}^{\gamma-1}\right] w_{x x}\right\} \\
&\qq +\alpha \partial_{t}^{j} \partial_{x}^{i-1}\left\{\sigma_{x}\left[\left(\tilde{\eta}_{x}+w_{x}\right)^{-\gamma}-\tilde{\eta}_{x}^{-\gamma}+\gamma \tilde{\eta}_{x}^{-\gamma-1} w_{x}\right]\right\},\\
&Q_3:=\mu \sum_{\ell=1}^j C_j^\ell \p^\ell_t(1+t)^{-\la}\p^{j+1-\ell}_t w.
\end{aligned}
\ee
 Summations $\sum_{\ell=1}^{i-1}$ and $\sum_{\ell=2}^{i-1}$ should be understood to be 0 when $i=1$ and $i=1,2,$ respectively. Multiply equation \eqref{ehelliptic} by $\tilde{\eta}_{x}^{\gamma+1} \sigma^{(\alpha+i-1) / 2}$ and perform the spatial $L^{2}$ -norm of the product to give
\bes
\bali
&\left\|\sigma^{\frac{\alpha+i+1}{2}} \partial_{t}^{j} \partial_{x}^{i+1} w+(\alpha+i) \sigma^{\frac{\alpha+i-1}{2}} \sigma_x \partial_{t}^{j} \partial_{x}^{i} w\right\|^{2} \\
\leq&(1+t)^{2(\la+1)}\left\|\sigma^{\frac{\alpha+i-1}{2}} \partial_{t}^{j+2} \partial_{x}^{i-1} w\right\|^{2}+(1+t)^{2}\left\|\sigma^{\frac{\alpha+i-1}{2}} \partial_{t}^{j+1} \partial_{x}^{i-1} w\right\|^{2} \\
&+(1+t)^{2(\la+1)}\left\|\sigma^{\frac{\alpha+i-1}{2}} (Q_{1},Q_{2},Q_{3}) \right\|^{2}.
\eali
\ees
Similar to the derivation of \eqref{eelliptic2}, we can then obtain
\be\label{ehelliptic7}
\bali
&(1+t)^{-2 j+\dl{\bf1}_{\la<1}}\mathcal{E}_{j,i} \\
\leq&(1+t)^{2(\la+1)}\left\|\sigma^{\frac{\alpha+i-1}{2}} \partial_{t}^{j+2} \partial_{x}^{i-1} w\right\|^{2}+(1+t)^{2}\left\|\sigma^{\frac{\alpha+i-1}{2}} \partial_{t}^{j+1} \partial_{x}^{i-1} w\right\|^{2} \\
&+(1+t)^{2(\la+1)}\left\|\sigma^{\frac{\alpha+i-1}{2}} (Q_{1},Q_{2},Q_{3}) \right\|^{2}.
\eali
\ee

We will use this estimate to prove the following lemma by mathematical induction.


\begin{lemma}\label{lhighe}
Under the assumption of \eqref{eenergyass} for suitably small positive number $\epsilon_0\in (0,1)$. Then for $j \geq 0, i \geq 1,$ and $0 \leq i+j \leq m$
\be\label{ehelliptic2}
\mathcal{E}_{j, i}(t) \lesssim \sum_{\ell=0}^{i+j} \mathcal{E}_{\ell}(t), \quad t \in[0, T].
\ee
\end{lemma}
\begin{proof}
We use induction on $i+j$ to prove this lemma. As shown in Lemma \ref{llowee}, we know that \eqref{ehelliptic2} holds for $i+j\leq1 .$ For $1 \leq k \leq m-1,$ we make the induction hypothesis that \eqref{ehelliptic2} holds for all $i\geq 1,j \geq 0,$ and $i+j \leq k,$ that is,
\be\label{ehelliptica}
 \mathcal{E}_{j, i}(t) \lesssim \sum_{\ell=0}^{i+j} \mathcal{E}_{\ell}(t), \quad i\geq 1,j \geq 0, i+j \leq k,
\ee
it then suffices to prove \eqref{ehelliptic2} for $i\geq1, j \geq 0,$ and $i+j=k+1 .$ We will bound $\mathcal{E}_{k+1-\ell, \ell}$ from $\ell=1$ to $k+1$ step by step.

We estimate $Q_{1}$ and $Q_{3}$ given by \eqref{ehelliptic5} as follows. For $Q_{1}$, it follows from \eqref{ecor3} and \eqref{ecor4} that
\bes
\begin{aligned}
\left|Q_{1}\right| \lesssim & \sum_{\ell=1}^{j}(1+t)^{-(\la+1)-\ell}\left(\sigma\left|\partial_{t}^{j-\ell} \partial_{x}^{i+1} w\right|+\left|\partial_{t}^{j-\ell} \partial_{x}^{i} w\right|\right) \\
&+\sum_{\ell=0}^{j} \sum_{r=1}^{i-1}(1+t)^{-(\la+1)-\ell}\left|\partial_{t}^{j-\ell} \partial_{x}^{r} w\right|,
\end{aligned}
\ees
and
\bes
\begin{aligned}
\left|Q_{3}\right| \lesssim & \sum_{\ell=1}^{j}(1+t)^{-\la-\ell}\left|\partial_{t}^{j+1-\ell} \partial_{x}^{i-1} w\right|.
\end{aligned}
\ees

So that we can get
\bes
\bali
&\left\|\sigma^{\frac{\alpha+i-1}{2}} Q_{1}\right\|^{2} \\
 \lesssim& \sum_{\ell=1}^{j}(1+t)^{-2(\la+1)-2\ell}\left(\left\|\sigma^{\frac{\alpha+i+1}{2}} \partial_{t}^{j-\ell} \partial_{x}^{i+1} w\right\|^{2}+\left\|\sigma^{\frac{\alpha+i-1}{2}} \partial_{t}^{j-\ell} \partial_{x}^{i} w\right\|^{2}\right)\\
&\quad+\sum_{\ell=0}^{j} \sum_{r=1}^{i-1}(1+t)^{-2(\la+1)-2 \ell}\left\|\sigma^{\frac{\alpha+i-1}{2}} \partial_{t}^{j-\ell} \partial_{x}^{r} w\right\|^{2} \\
\lesssim&(1+t)^{-2 j-2(\la+1)+\dl{\bf1}_{\la<1}}\left(\sum_{\ell=1}^{j} \mathcal{E}_{j-\ell, i}+\sum_{\ell=0}^{j} \sum_{r=1}^{i-1} \mathcal{E}_{j-\ell, r}\right).
\eali
\ees
Here $\sum_{r=1}^{i-1}$ is understood to be $0$ if $i=1$.
And
\bes
\bali
&\left\|\sigma^{\frac{\alpha+i-1}{2}} Q_{3}\right\|^{2} \\
 \lesssim& \sum_{\ell=1}^{j}(1+t)^{-2\la-2\ell}\left\|\sigma^{\frac{\alpha+i-1}{2}} \partial_{t}^{j+1-\ell} \partial_{x}^{i-1} w\right\|^{2}\\
\lesssim&(1+t)^{-2 j-2(\la+1)+\dl{\bf1}_{\la<1}}\sum_{\ell=1}^{j}\left( \mathcal{E}_{j+1-\ell}{\bf 1}_{i=1}+ \mathcal{E}_{j+1-\ell,i-2}{\bf 1}_{i\geq 2}\right).
\eali
\ees
For $Q_{2},$ it follows from \eqref{ecor3}, \eqref{ecor4}, and \eqref{esobolev} that
\[
\begin{aligned}
&\left|Q_{2}\right|\\
  \lesssim& \sum_{n=0}^{j} \sum_{\ell=0}^{i-1} K_{n \ell}\left(\left|\partial_{t}^{j-n} \partial_{x}^{i-1-\ell}\left(\sigma w_{x x}\right)\right|+\left|\partial_{t}^{j-n} \partial_{x}^{i-1-\ell}\left(\sigma_x w_{x}\right)\right|\right) \\
 \lesssim& \sum_{n=0}^{j} \sum_{\ell=0}^{i-1} K_{n \ell}\left(\left|\sigma \partial_{t}^{j-n} \partial_{x}^{i-\ell+1} w\right|+\left|\sigma_x \partial_{t}^{j-n} \partial_{x}^{i-\ell} w\right|+\left| \partial_{t}^{j-n} \partial_{x}^{i-\ell-1} w\right|\right)\\
:= &\sum_{n=0}^{j} \sum_{\ell=0}^{i-1} Q_{2 n \ell}.
\end{aligned}
\]

Here the main term of $K_{n\ell}$ is $\p^n_t\p^\ell_x\lt(\t{\eta}^{-\g-2}_x w_x\rt)$. Then, again by using \eqref{ecor3}, \eqref{ecor4}, and \eqref{esobolev}, we have
\be\label{ehelliptic6}
\begin{array}{c}
K_{00}=\s{\epsilon_{0}}(1+t)^{-(\la+1)-\frac{\la+1}{\gamma+1}+\f{\dl}{2}{\bf1}_{\la<1}}, \\
K_{10}=\s{\epsilon_{0}}(1+t)^{-(\la+1)-\frac{\la+1}{\gamma+1}-1+\f{\dl}{2}{\bf1}_{\la<1}},\\
K_{01}=(1+t)^{-(\la+1)-\frac{\la+1}{\gamma+1}}\left|\partial_{x}^{2} w\right|, \\
K_{20}=\s{\epsilon_{0}}(1+t)^{-(\la+1)-\frac{\la+1}{\gamma+1}-2+\f{\dl}{2}{\bf1}_{\la<1}}+(1+t)^{-(\la+1)-\frac{\la+1}{\gamma+1}}\left|\partial_{t}^{2} \partial_{x} w\right|, \\
K_{11}=(1+t)^{-(\la+1)-\frac{\la+1}{\gamma+1}-1}\left|\partial_{x}^{2} w\right|+(1+t)^{-(\la+1)-\frac{\la+1}{\gamma+1}}\left|\partial_{t} \partial_{x}^{2} w\right|, \\
K_{02}=(1+t)^{-(\la+1)-\frac{\la+1}{\gamma+1}}\left|\partial_{x}^{3} w\right|.
\end{array}
\ee

We do not list here $K_{n \ell}$ for $n+\ell \geq 3$ since we can use the same method to estimate $Q_{2 n \ell}$ for $n+\ell\geq 3$ as that for $n+\ell \leq 2 .$

First, for $n=\ell=0$,
\bes
\bali
&\left\|{\sigma}^{\frac{\alpha+i-1}{2}} Q_{200}\right\|^{2}\\
\ls& \epsilon_{0}(1+t)^{-2(\la+1)-\frac{2(\la+1)}{\gamma+1}+\dl{\bf1}_{\la<1}}\lt(\left\|{\sigma}^{\frac{\alpha+i+1}{2}}\p^{j}_t\p^{i+1}_x w\right\|^{2}\rt.\\
       &\lt.+\left\|{\sigma}^{\frac{\alpha+i-1}{2}}\sigma_x\p^{j}_t\p^{i}_x w\right\|^{2}+\left\|{\sigma}^{\frac{\alpha+i-1}{2}}\p^{j}_t\p^{i-1}_x w\right\|^{2}\rt)\\
& \ls\e_0(1+t)^{-2 j-2(\la+1)-\frac{2(\la+1)}{\gamma+1}+2\dl{\bf1}_{\la<1}}\lt(\mathcal{E}_{j,i}+\mathcal{E}_{j,i-2}{\bf1}_{i\geq2}+\mathcal{E}_{j}{\bf1}_{i=1}\rt),\\
&\ls \e_0(1+t)^{-2 j-2(\la+1)+\dl{\bf1}_{\la<1}}\lt(\mathcal{E}_{j,i}+\mathcal{E}_{j,i-2}{\bf1}_{i\geq2}+\mathcal{E}_{j}{\bf1}_{i=1}\rt),
\eali
\ees
where at the last but second line we have used the fact $\sigma_x$ is bounded and at last line, used that $\dl\in(0,\f{2(\la+1)}{\g+1})$.

Also, case $n=1,\ell=0$ can be estimated the same as that for $n=0,\ell=0$,
\bes
\bali
&\left\|{\sigma}^{\frac{\alpha+i-1}{2}} Q_{210}\right\|^{2}\\
\ls& \epsilon_{0}(1+t)^{-2(\la+1)-\frac{2(\la+1)}{\gamma+1}-2+\dl{\bf1}_{\la<1}}\lt(\left\|{\sigma}^{\frac{\alpha+i+1}{2}}\p^{j-1}_t\p^{i+1}_x w\right\|^{2}\rt.\\
       &\lt.+\left\|{\sigma}^{\frac{\alpha+i-1}{2}}\sigma_x\p^{j-1}_t\p^{i}_x w\right\|^{2}+\left\|{\sigma}^{\frac{\alpha+i-1}{2}}\p^{j-1}_t\p^{i-1}_x w\right\|^{2}\rt)\\
 \ls&\e_0(1+t)^{-2 j-2(\la+1)+\dl{\bf1}_{\la<1}}\lt(\mathcal{E}_{j-1,i}+\mathcal{E}_{j-1,i-2}{\bf1}_{i\geq2}+\mathcal{E}_{j-1}{\bf1}_{i=1}\rt).
\eali
\ees

For $n=0,\ell=1$, remember here that $i\geq2$, then we have
\bes
\bali
&\left\|{\sigma}^{\frac{\alpha+i-1}{2}} Q_{201}\right\|^{2}\\
\ls& (1+t)^{-2(\la+1)-\frac{2(\la+1)}{\gamma+1}}\| \sigma^{1/2}w_{xx} \|^{2}_{L^\i}\lt(\left\|{\sigma}^{\frac{\alpha+i}{2}}\p^{j}_t\p^{i}_x w\right\|^{2}\rt.\\
       &\lt.+\left\|{\sigma}^{\frac{\alpha+i-2}{2}}\sigma_x\p^{j}_t\p^{i-1}_x w\right\|^{2}+\left\|{\sigma}^{\frac{\alpha+i-2}{2}}\p^{j}_t\p^{i-2}_x w\right\|^{2}\rt)\\
& \ls\e_0(1+t)^{-2 j-2(\la+1)+\dl{\bf1}_{\la<1}}\lt(\mathcal{E}_{j,i-1}+\mathcal{E}_{j,i-3}{\bf1}_{i\geq3}+\mathcal{E}_{j}{\bf1}_{i=2}\rt).
\eali
\ees
Next we go to estimate the terms with $n+\ell=2$. For $n=2,\ell=0$, substitute \eqref{ehelliptic6} into $Q_{220}$ to obtain that
\be\label{ehe0}
\bali
&\left\|{\sigma}^{\frac{\alpha+i-1}{2}} Q_{220}\right\|^{2}\\
\ls&(1+t)^{-2(\la+1)-\frac{2(\la+1)}{\gamma+1}}\| \sigma^{1/2} w_{xtt} \|^{2}_{L^\i}\lt(\left\|{\sigma}^{\frac{\alpha+i}{2}}\p^{j-2}_t\p^{i+1}_x w\right\|^{2}\rt.\\
       &\lt.+\left\|{\sigma}^{\frac{\alpha+i-2}{2}}\sigma_x\p^{j-2}_t\p^{i}_x w\right\|^{2}+\left\|{\sigma}^{\frac{\alpha+i-2}{2}}\p^{j-2}_t\p^{i-1}_x w\right\|^{2}\rt)\\
&+\epsilon_{0}(1+t)^{-4-\frac{2(\la+1)}{\gamma+1}-2(\la+1)+\dl{\bf1}_{\la<1}}\lt(\left\|{\sigma}^{\frac{\alpha+i+1}{2}}\p^{j-2}_t\p^{i+1}_x w\right\|^{2}\rt.\\
       &\lt.+\left\|{\sigma}^{\frac{\alpha+i-1}{2}}\sigma_x\p^{j-2}_t\p^{i}_x w\right\|^{2}+\left\|{\sigma}^{\frac{\alpha+i-1}{2}}\p^{j-2}_t\p^{i-1}_x w\right\|^{2}\rt).
\eali
\ee
 Remember the Hardy inequality \eqref{ehardy}, then we can obtain
\be\label{ehe1}
\bali
&\left\|{\sigma}^{\frac{\alpha+i-2}{2}}\sigma_x\p^{j-2}_t\p^{i}_x w\right\|^{2}\\
\ls& \left\|{\sigma}^{\frac{\alpha+i}{2}}\p^{j-2}_t\p^{i}_x w\right\|^{2}+ \left\|{\sigma}^{\frac{\alpha+i}{2}}\p^{j-2}_t\p^{i+1}_x w\right\|^{2} .
\eali
\ee
So, Inserting \eqref{ehe1} into \eqref{ehe0}, we get
\bes
\bali
&\left\|{\sigma}^{\frac{\alpha+i-1}{2}} Q_{220}\right\|^{2}\\
\ls&\e_0(1+t)^{-4-\frac{2(\la+1)}{\gamma+1}-2(\la+1)+\dl{\bf1}_{\la<1}}\lt(\left\|{\sigma}^{\frac{\alpha+i}{2}}\p^{j-2}_t\p^{i+1}_x w\right\|^{2}\rt.\\
       &\lt.+\left\|{\sigma}^{\frac{\alpha+i}{2}}\p^{j-2}_t\p^{i}_x w\right\|^{2}+\left\|{\sigma}^{\frac{\alpha+i-2}{2}}\p^{j-2}_t\p^{i-1}_x w\right\|^{2}\rt)\\
&+\epsilon^2_{0}(1+t)^{-4-\frac{2(\la+1)}{\gamma+1}-2(\la+1)+\dl{\bf1}_{\la<1}}\lt(\left\|{\sigma}^{\frac{\alpha+i+1}{2}}\p^{j-2}_t\p^{i+1}_x w\right\|^{2}\rt.\\
       &\lt.+\left\|{\sigma}^{\frac{\alpha+i-1}{2}}\sigma_x\p^{j-2}_t\p^{i}_x w\right\|^{2}+\left\|{\sigma}^{\frac{\alpha+i-1}{2}}\p^{j-2}_t\p^{i-1}_x w\right\|^{2}\rt)\\
 \ls&\e_0(1+t)^{-2 j-2(\la+1)-\frac{2(\la+1)}{\gamma+1}+2\dl{\bf1}_{\la<1}}\sum^{i+1}_{\ell=i-2}\mathcal{E}_{j-2,\ell}\\
\ls&\e_0(1+t)^{-2 j-2(\la+1)+\dl{\bf1}_{\la<1}}\sum^{i+1}_{\ell=i-2}\mathcal{E}_{j-2,\ell}.
\eali
\ees

In the case $n=\ell=1$, we have $i\geq 2$.
\bes
\bali
&\left\|{\sigma}^{\frac{\alpha+i-1}{2}} Q_{211}\right\|^{2}\\
\ls&(1+t)^{-2(\la+1)-\frac{2(\la+1)}{\gamma+1}}\|\sigma w_{xxt} \|^{2}_{L^\i}\lt(\left\|{\sigma}^{\frac{\alpha+i-1}{2}}\p^{j-1}_t\p^{i}_x w\right\|^{2}\rt.\\
       &\lt.+\left\|{\sigma}^{\frac{\alpha+i-3}{2}}\sigma_x\p^{j-1}_t\p^{i-1}_x w\right\|^{2}+\left\|{\sigma}^{\frac{\alpha+i-3}{2}}\p^{j-1}_t\p^{i-2}_x w\right\|^{2}\rt)\\
&+(1+t)^{-2(\la+1)-2-\frac{2(\la+1)}{\gamma+1}}\|\sigma^{1/2} w_{xx} \|^{2}_{L^\i}\lt(\left\|{\sigma}^{\frac{\alpha+i}{2}}\p^{j-1}_t\p^{i}_x w\right\|^{2}\rt.\\
       &\lt.+\left\|{\sigma}^{\frac{\alpha+i-2}{2}}\sigma_x\p^{j-1}_t\p^{i-1}_x w\right\|^{2}+\left\|{\sigma}^{\frac{\alpha+i-2}{2}}\p^{j-2}_t\p^{i-2}_x w\right\|^{2}\rt)\\
\ls&\e_0(1+t)^{-2j-2(\la+1)+\dl{\bf1}_{\la<1}}\lt(\mathcal{E}_{j-1,i}+\mathcal{E}_{j-1,i-2}\rt)\\
 &+\e_0(1+t)^{-2 j-2(\la+1)+\dl{\bf1}_{\la<1}}\lt(\mathcal{E}_{j-1,i-1}+\mathcal{E}_{j-1,i-2}\rt)\\
\ls&\e_0(1+t)^{-2 j-2(\la+1)+\dl{\bf1}_{\la<1}}\sum^{i}_{\ell=i-2}\mathcal{E}_{j-1,\ell}.
\eali
\ees
Here due to $i\geq 2$ and \eqref{ehardy}, we have used
\bes
\left\|{\sigma}^{\frac{\alpha+i-3}{2}}\p^{j-1}_t\p^{i-1}_x w\right\|^{2}\les  \left\|{\sigma}^{\frac{\alpha+i-1}{2}}\p^{j-1}_t\p^{i}_x w\right\|^{2}\left\|{\sigma}^{\frac{\alpha+i-1}{2}}\p^{j-1}_t\p^{i-1}_x w\right\|^{2}.
\ees
In the case $n=0,\ \ell=2$, we have $i\geq 3$.

\bes
\bali
&\left\|{\sigma}^{\frac{\alpha+i-1}{2}} Q_{202}\right\|^{2}\\
\ls&(1+t)^{-2(\la+1)-\frac{2(\la+1)}{\gamma+1}}\|\sigma^{3/2} w_{xxx} \|^{2}_{L^\i}\lt(\left\|{\sigma}^{\frac{\alpha+i-2}{2}}\p^{j}_t\p^{i-1}_x w\right\|^{2}\rt.\\
       &\lt.+\left\|{\sigma}^{\frac{\alpha+i-4}{2}}\sigma_x\p^{j}_t\p^{i-2}_x w\right\|^{2}+\left\|{\sigma}^{\frac{\alpha+i-4}{2}}\p^{j}_t\p^{i-3}_x w\right\|^{2}\rt)\\
\ls&\e_0(1+t)^{-2j-2(\la+1)+\dl{\bf1}_{\la<1}}\lt(\mathcal{E}_{j,i-2}+\mathcal{E}_{j,i-3}\rt).
\eali
\ees
Since the leading term of $K_{n\ell}$ is
\bes
\sum^n_{\iota=0}(1+t)^{-(\la+1)-\f{\la+1}{\g+1}-\iota}|\p^{n-\iota}_t\p^{\ell+1}_x w|,
\ees
other terms for $n+\ell\geq 3$ can be handled with the same line.

Now combining all the above estimates for $Q_1,\ Q_2$ and $Q_3$, we get
\bes
\bali
&\left\|{\sigma}^{\frac{\alpha+i-1}{2}} (Q_1,Q_2,Q_3)\right\|^{2}\\
\ls& (1+t)^{-2 j-2(\la+1)+\dl{\bf1}_{\la<1}}\lt(\e_0\mathcal{E}_{j,i}+\sum^j_{\ell=0} \mathcal{E}_{\ell}+\sum_{0\leq \ell\leq j\atop \ell+r\leq i+j-1} \mathcal{E}_{\ell,r}\rt).
\eali
\ees
Substituting this into \eqref{ehelliptic7}, we get
\be\label{ehelliptic3}
\bali
\mathcal{E}_{j,i} \les& (1+t)^{2 j+2(\la+1)-\dl{\bf1}_{\la<1}}\left\|\sigma^{\frac{\alpha+i-1}{2}}\p^{j+2}_t\p^{i-1}_x w\right\|^{2}\\
&+(1+t)^{2 (j+1)-\dl{\bf1}_{\la<1}}\left\|{\sigma}^{\frac{\alpha+i-1}{2}} \p^{j+1}_t\p^{i-1}_x w\right\|^{2}\\
&+\e_0\mathcal{E}_{j,i}+\sum_{0\leq \ell\leq j\atop \ell+r\leq i+j-1}\mathcal{E}_{\ell,r}+\sum_{\ell=0}^{j} \mathcal{E}_{\ell}.
\eali
\ee
In particularly, when $i\geq 2$, we have
\be\label{ehelliptic4}
\bali
\mathcal{E}_{j,i} \les&\mathcal{E}_{j+2,i-2}+\mathcal{E}_{j+1,i-2}+\sum_{0\leq \ell\leq j\atop \ell+r\leq i+j-1}\mathcal{E}_{\ell,r}+\sum_{\ell=0}^{j} \mathcal{E}_{\ell}.
\eali
\ee
In what follows, we use \eqref{ehelliptic4} and the induction hypothesis \eqref{ehelliptica} to show that \eqref{ehelliptic2} holds for $i+j=k+1 .$ First, choosing $j=k$ and $i=1$ in \eqref{ehelliptic3} gives
$$
\mathcal{E}_{k, 1}(t) \lesssim \mathcal{E}_{k+1}(t)+\sum_{0\leq \ell\leq k\atop \ell+r\leq k}\mathcal{E}_{\ell,r}+\sum_{\ell=0}^{k} \mathcal{E}_{\ell},
$$
which, together with \eqref{ehelliptica} implies
\bes
 \mathcal{E}_{k, 1}(t) \lesssim \sum_{\ell=0}^{k+1} \mathcal{E}_{\ell}(t).
\ees
Similarly, using \eqref{ehelliptic4}, we have
$$
\begin{aligned}
\mathcal{E}_{k-1,2}(t) \lesssim & \mathcal{E}_{k+1}(t)+\mathcal{E}_{k}(t)+\sum_{0\leq \ell\leq k-1\atop \ell+r\leq k} \mathcal{E}_{\ell,r}+\sum_{\ell=0}^{k-1} \mathcal{E}_{\ell}\lesssim \sum_{\ell=0}^{k+1} \mathcal{E}_{\ell}(t).
\end{aligned}
$$
For $\mathcal{E}_{k-2,3},$ it follows from \eqref{ehelliptic4} and \eqref{ehelliptica} that
\bes
\bali
\mathcal{E}_{k-2,3}(t) \lesssim& \mathcal{E}_{k, 1}(t)+\mathcal{E}_{k-1,1}(t)+\sum_{0\leq \ell\leq k-2\atop \ell+r\leq k}+ \mathcal{E}_{\ell,r}
+\sum_{\ell=0}^{k-2} \mathcal{E}_{\ell}\lesssim& \sum_{\ell=0}^{k+1} \mathcal{E}_{\ell}(t).
\eali
\ees
The other cases can be handled similarly. So we have proved \eqref{ehelliptic2} when $i+j=k+1 .$ This finishes the proof of Lemma \ref{lhighe}.
\end{proof}

\subsection{Nonlinear Weighted Energy Estimates}\label{subweight}

In this subsection, we prove that the weighted energy $\mathcal{E}_{j}(t)$ can be bounded by the initial data for $t \in[0, T]$.
 \begin{proposition}\label{proweightedenergy}
 Suppose that \eqref{eenergyass} holds for a suitably small positive number $\epsilon_{0} \in(0,1) .$ Then for $t \in[0, T]$
\bes
 \mathcal{E}_{j}(t) \lesssim \sum_{\ell=0}^{j} \mathcal{E}_{\ell}(0), \quad j=0,1, \ldots, m.
\ees
\end{proposition}
The proof of Proposition \ref{proweightedenergy} contains Lemma \ref{lbasice} and Lemma \ref{ehweightede} below.

\noindent {\bf Basic Energy Estimates}
\begin{lemma}\label{lbasice}
 Suppose that \eqref{eenergyass} holds for a suitably small positive number $\epsilon_{0} \in$ $(0,1) .$ Then
\be\label{ebasicenergy}
\bali
 &\mathcal{E}_{0}(t)+\int_{0}^{t} \int\lt[(1+\tau)^{1-\dl{\bf 1}_{\la<1}} \sigma^\al w_{\tau}^{2}+ (1+\tau)^{-1-\dl{\bf 1}_{\la<1}}\sigma^{\al+1} w_{x}^{2}\rt] d x d \tau\\
 \lesssim& \mathcal{E}_{0}(0), \quad t \in[0, T].
\eali
\ee
\end{lemma}

\begin{proof}
 In order to simplify the presentation, by using Taylor expansion and smallness of $\t{\eta}_xw_x$, we rewrite  $\eqref{eerr}_1$ as follows

\be\label{ebasice1}
\sigma^\al w_{tt}+\f{\mu}{(1+t)^\la}\sigma^{\al} w_t-[\sigma^{\al+1}\t{\eta}^{-\g-1}_x(1+o(1))w_x]_x=0,
\ee
 where $o(1)$ means $o(1)\ls \s{\e_0}$.

 The proof will be divided into two parts. One is for $0<\la<1,\mu>0$ and the other is for $\la=1,\mu>2$.

\noindent{\bf Case 1: $\boldsymbol{0<\la<1,\mu>0}$}

 Multiplying $\eqref{ebasice1}$ by $(K+t)^\la w_{t}$, where $K>1$ is a suitably large constant, to be determined later, and integrating the product with respect to the spatial variable, then we can get
\bes
\bali
&\frac{1}{2}\frac{d}{d t} \int  \sigma^\al (K+t)^\la w_{t}^{2} d x- \f{\la}{2}(K+t)^{\la-1}\int \sigma^\al w_{t}^{2} d x +\mu\lt(\f{K+t}{1+t}\rt)^\la\int \sigma^\al w_{t}^{2} d x\\
&+(K+t)^\la\t{\eta}^{-\g-1}_x\int\sigma^{\al+1}\left[(1+o(1))w_x\right] w_{x t} d x=0.
\eali
\ees
We then have
\be\label{ebasic1l}
\bali
&\frac{1}{2}\frac{d}{d t} \int (K+t)^\la \lt[ \sigma^\al w_{t}^{2}+ \t{\eta}^{-\g-1}_x(1+o(1))\sigma^{\al+1} w_{t}^{2}\rt] d x\\
&+\lt[\mu\lt(\f{K+t}{1+t}\rt)^\la- \f{\la}{2}(K+t)^{\la-1} \rt]\int \sigma^\al w_{t}^{2} d x\\
&-(1+o(1))\f{1}{2}\p_t\lt((K+t)^\la\t{\eta}^{-\g-1}_x\rt)\int \sigma^{\al+1} w_{x}^{2} d x=0.
\eali
\ee
Using the fact that $\t{\eta}_{xt}\geq 0$, we simplify \eqref{ebasic1l} as
\be\label{ebasic1c}
\bali
&\frac{1}{2}\frac{d}{d t} \int (K+t)^\la \lt[ \sigma^\al w_{t}^{2}+ \t{\eta}^{-\g-1}_x(1+o(1))\sigma^{\al+1} w_{t}^{2}\rt] d x\\
&+\lt[\mu- \f{\la}{2}(K+t)^{\la-1} \rt]\int \sigma^\al w_{t}^{2} d x\\
&-(1+o(1))\f{\la}{2}(K+t)^{\la-1}\t{\eta}^{-\g-1}_x\int \sigma^{\al+1} w_{x}^{2} d x\leq0.
\eali
\ee
Now multiplying $\eqref{ebasice1}$ by $\nu w$ for some small $\nu>0$, to be determined later, and integrating the product with respect to the spatial variable, then we can get
\be\label{ebasic2l}
\bali
&\nu \frac{d}{d t} \int  \sigma^{\al}w_{t}w d x- \nu\int \sigma^\al w_{t}^{2} d x +\f{\nu\mu}{2}\p_t\int  \f{1}{(1+t)^\la}\sigma^\al w^{2} d x\\
&+\f{\nu\mu\la}{2(1+t)^{\la+1}}\int  \sigma^\al w^{2} d x+\nu \t{\eta}^{-\g-1}_x(1+o(1))\int \sigma^{\al+1}w^2_x d x=0.
\eali
\ee
Adding \eqref{ebasic1c} and \eqref{ebasic2l}, we have
\be\label{ebasic4}
\bali
&\frac{d}{d t} \int \t{\mathfrak{E}}_0(x,t)d x+\f{\nu\mu\la}{2(1+t)^{\la+1}}\int  \sigma^\al w^{2} d x\\
&+\lt[\mu- \f{\la}{2}(K+t)^{\la-1}-\nu \rt]\int \sigma^\al w_{t}^{2} d x\\
&+(1+o(1))\t{\eta}^{-\g-1}_x \lt(\nu-\f{\la}{2}(K+t)^{\la-1}\rt)\int \sigma^{\al+1} w_{x}^{2} d x\leq0.
\eali
\ee Here
\bes
\bali
\t{\mathfrak{E}}_0(x,t):=&  \f{(K+t)^\la}{2} \lt[ \sigma^\al w_{t}^{2}+ (1+o(1))\t{\eta}^{-\g-1}_x\sigma^{\al+1} w_{x}^{2}\rt] \\
&+  \nu\sigma^{\al}w_{t}w+\f{\nu\mu}{2(1+t)^\la} \sigma^\al w^{2}.
\eali
\ees
By using Cauchy-Schwartz inequality, we have
\be\label{ebasic3r}
\bali
\f{(K+t)^\la}{4} &\lt[ \sigma^\al w_{t}^{2}+\t{\eta}^{-\g-1}_x\sigma^{\al+1} w_{x}^{2}\rt]+\lt(\f{\nu\mu}{2}-\nu^2\rt)\f{1}{(1+t)^\la} \sigma^\al w^{2} \\
                                              &\qq\qq\qq \leq \t{\mathfrak{E}}_0\leq \\
\f{3(K+t)^\la}{4} &\lt[ \sigma^\al w_{t}^{2}+\t{\eta}^{-\g-1}_x\sigma^{\al+1} w_{x}^{2}\rt]+\lt(\f{\nu\mu}{2}+\nu^2\rt)\f{1}{(1+t)^\la} \sigma^\al w^{2}. \\
\eali
\ee
Since $\la<1$, by first choosing small $\nu$ and then large $K$, we can get
\be\label{ebasic4r}
\bali
&\frac{d}{d t} \int \t{\mathfrak{E}}_0(x,t)d x+\f{\nu\mu\la}{2(1+t)^{\la+1}}\int  \sigma^\al w^{2} d x\\
&+\f{\nu}{2}\int \sigma^\al w_{t}^{2} d x+\f{\nu}{2}\t{\eta}^{-\g-1}_x\int \sigma^{\al+1} w_{x}^{2} d x\leq 0.
\eali
\ee
Now multiplying \eqref{ebasic4r} by $(K+t)^{\la-\dl}$, we can achieve
\be\label{ebasic4r1}
\bali
&\frac{d}{d t} \int (K+t)^{\la-\dl}\t{\mathfrak{E}}_0(x,t)d x-(\la-\dl)(K+t)^{\la-1-\dl}\t{\mathfrak{E}}_0(x,t)\\
&+\f{\nu\mu\la(K+t)^{\la-\dl}}{2(1+t)^{\la+1}}\int  \sigma^\al w^{2} d x\\
&+\f{\nu(K+t)^{\la-\dl}}{2}\lt\{\int \sigma^\al w_{t}^{2} d x+\t{\eta}^{-\g-1}_x\int \sigma^{\al+1} w_{x}^{2} d x\rt\}\leq 0.
\eali
\ee
By inserting \eqref{ebasic3r} into \eqref{ebasic4r1}, we have
\bes
\bali
&\frac{d}{d t} \int (K+t)^{\la-\dl}\t{\mathfrak{E}}_0(x,t)d x\\
&+\f{\nu\mu(K+t)^{\la-\dl}}{2(1+t)^{\la+1}}\underbrace{\lt(\la-(\la-\dl)(1+\f{2\nu}{\mu})\rt)}_{L_1}\int  \sigma^\al w^{2} d x\\
&+(K+t)^{\la-\dl}\underbrace{\lt(\f{\nu}{2}-\f{3}{4}(K+t)^{\la-1}\rt)}_{L_2}\lt\{\int \sigma^\al w_{t}^{2} d x+\t{\eta}^{-\g-1}_x\int \sigma^{\al+1} w_{x}^{2} d x\rt\}\leq 0.
\eali
\ees
Again, by choosing small $\nu$ and large $K$, for any $\dl>0$, we can assure that $L_1$ and $L_2$ are positive. Then we have for some constant $c_{\la,\mu}$

\be\label{ebasic4r3}
\bali
&\frac{d}{d t} \int (K+t)^{\la-\dl}\t{\mathfrak{E}}_0(x,t)d x\\
&+c_{\la,\mu}(K+t)^{\la-\dl}\lt\{\int \sigma^\al w_{t}^{2} d x+\t{\eta}^{-\g-1}_x\int \sigma^{\al+1} w_{x}^{2} d x\rt\}\leq 0.
\eali
\ee
Now we multiply \eqref{ebasic1c} by $(K+t)^{1-\dl}$ to achieve that
\be\label{ebasic4r4}
\bali
&\frac{1}{2}\frac{d}{d t} \int (K+t)^{1+\la-\dl} \lt[ \sigma^\al w_{t}^{2}+ (1+o(1))\t{\eta}^{-\g-1}_x\sigma^{\al+1} w_{t}^{2}\rt] d x\\
&+c_{\la,\mu}  (K+t)^{1-\dl}\int \sigma^\al w_{t}^{2} d x\\
&-c_{\la,\mu}  (K+t)^{\la-\dl}\t{\eta}^{-\g-1}_x\int \lt(\sigma^{\al+1} w_{x}^{2}+\sigma^{\al+1} w_{x}^{2}\rt)d x\leq 0.
\eali
\ee
Multiplying a small number $\nu_1$ to \eqref{ebasic4r4} and then adding the resulting equations to \eqref{ebasic4r3}, we can get
\be\label{ebasic4r5}
\bali
&\frac{d}{d t} \int {\mathfrak{E}}_0(x,t)d x\\
&\q+c_{\la,\mu}(1+t)^{1-\dl}\int \sigma^\al w_{t}^{2} d x+c_{\la,\mu}(1+t)^{-1-\dl}\int \sigma^{\al+1} w_{x}^{2} d x\leq 0, \eali
\ee
where
\bes
\bali
{\mathfrak{E}}_0(x,t)&:=(K+t)^{\la-\dl}\t{\mathfrak{E}}_0(x,t)\\
 &\q +\nu_1(K+t)^{1+\la-\dl} \lt[ \sigma^\al w_{t}^{2}+ (1+o(1))\t{\eta}^{-\g-1}_x\sigma^{\al+1} w_{x}^{2}\rt]\\
 & \thickapprox  (1+t)^{1+\la-\dl }\sigma^\al w_{t}^{2}+(1+t)^{-\dl }\sigma^{\al+1} w_{x}^{2}+ (1+t)^{-\dl }\sigma^\al w^2,
\eali
\ees
\bes
\int{\mathfrak{E}}_0(x,t)dx\thickapprox \mathcal{E}_0(t).
\ees Here we have used the fact that $\t{\eta}^{-\g-1}_x\thickapprox (1+t)^{-(\la+1)}$.

Now integrating \eqref{ebasic4r5} with respect to time variable from $0$ to $t$. we get \eqref{ebasicenergy} in the case of $0<\la<1,\mu>0$.

\noindent{\bf Case 2: $\boldsymbol{\la=1,\mu>2}$}

Multiplying $\eqref{ebasice1}$ by $(1+t)^2w_{t}$ and integrating the product with respect to the spatial variable, then we can get
\bes
\bali
&\frac{1}{2}\frac{d}{d t} \int  \sigma^\al (1+t)^2w_{t}^{2} d x- (1+t)\int \sigma^\al w_{t}^{2} d x +\mu(1+t)\int \sigma^\al w_{t}^{2} d x\\
&+(1+t)^2\t{\eta}^{-\g-1}_x \int\sigma^{\al+1}\left[(1+o(1))w_x\right] w_{x t} d x=0.
\eali
\ees

By using that $\t{\eta}_{xt}\geq 0$, we then have
\be\label{ebasic1}
\bali
&\frac{1}{2}\frac{d}{d t} \int (1+t)^2 \sigma^\al w_{t}^{2}+ (1+t)^2\t{\eta}^{-\g-1}_x(1+o(1))\sigma^{\al+1} w_{t}^{2}d x\\
&+(\mu-1) (1+t)\int \sigma^\al w_{t}^{2} d x\\
&-(1+o(1))(1+t) \t{\eta}^{-\g-1}_x\int \sigma^{\al+1} w_{x}^{2} d x=0.
\eali
\ee

Now multiplying \eqref{ebasice1} by $\nu(1+t)w$ for some positive $\nu$ to be determined later, and integrating the product with respect to the spatial variable, then we can get
\be\label{ebasic2}
\bali
&\nu \frac{d}{d t} \int  \sigma^{\al}(1+t)w_{t}w d x- \nu (1+t)\int \sigma^\al w_{t}^{2} d x +\f{\kappa(\mu-1)}{2}\p_t\int  \sigma^\al w^{2} d x\\
&+\nu(1+t)\t{\eta}^{-\g-1}_x\int \sigma^{\al+1}(1+o(1))w^2_x d x=0. \eali
\ee
Adding \eqref{ebasic1} and \eqref{ebasic2}, we have

\be\label{ebasic4}
\bali
&\f{d}{dt}\int{\mathfrak{E}}_0(t)dx+(\mu-1-\nu) (1+t)\int \sigma^\al w_{t}^{2} d x\\
&+(\nu-1)(1+t)\t{\eta}^{-\g-1}_x(1+o(1))\int \sigma^{\al+1}w^2_x d x\leq 0.
\eali
\ee
Here
\bes
\bali
{\mathfrak{E}}_0(x,t):=&  \f{(1+t)^2}{2}\lt[ \sigma^\al w_{t}^{2}+\t{\eta}^{-\g-1}_x(1+o(1))\sigma^{\al+1} w^2_x\rt] \\
&+ \nu (1+t)  \sigma^\al w_{t}w +  \f{\nu(\mu-1)}{2}\sigma w^2.
\eali
\ees
Now, since $\mu>2$, we assume $\mu=2+2\kappa$ for some positive $\kappa$. Choosing $\nu=1+\kappa$, we can achieve 

\bes
\bali
{\mathfrak{E}}_0(x,t):=&  \f{(1+t)^2}{2}\lt[ \sigma^\al w_{t}^{2}+\t{\eta}^{-\g-1}_x(1+o(1))\sigma^{\al+1} w^2_x\rt] \\
&+ (1+\kappa) (1+t)  \sigma^\al w_{t}w +  \f{(1+\kappa)(1+2\kappa)}{2}\sigma^\al w^2.
\eali
\ees
By using Cauchy-Shwartz inequality to absorb the term involving $w_{t}w$ and remembering that $\t{\eta}^{-\g-1}_{x}\thickapprox (1+t)^{-2}$, it is not hard to deduce that
\bes
{\mathfrak{E}}_0(x,t)\thickapprox  (1+t)^2\sigma^\al w_{t}^{2}+\sigma^{\al+1} w_{x}^{2}+ \sigma^\al w^2,\q \int{\mathfrak{E}}_0(x,t)dx\thickapprox \mathcal{E}_0(t).
\ees
Then \eqref{ebasic4} becomes
\be\label{ebasic3}
\bali
&\f{d}{dt}\int{\mathfrak{E}}_0(t)dx+\kappa (1+t)\int \sigma^\al w_{t}^{2} d x\\
&+\kappa(1+t)\t{\eta}^{-\g-1}_x(1+o(1))\int \sigma^{\al+1}w^2_x d x\leq 0.
\eali
\ee
Now integrating \eqref{ebasic3} with respect to time variable from $0$ to $t$. we get \eqref{ebasicenergy} in the case of $\la=1,\ \mu>2$.

\end{proof}

\noindent {\bf Higher-Order Energy Estimates}

 For  $k \geq 1$, $\partial_{t}^{k}\eqref{eerr}_1$  yields that
\be\label{ehweighted0}
\begin{aligned}
&\sigma^\al \partial_{t}^{k+2} w+\f{\mu}{(1+t)^\la}\sigma^\al \partial_{t}^{k+1} w+\mu\sigma^\al\sum^k_{\ell=1}C^\ell_k\p^\ell_t(1+t)^{-\la}\p^{k+1-\ell}_t w\\
&\q -\left[\sigma^{\al+1}\left(\t{\eta}_x+w_{x}\right)^{-\gamma-1} \partial_{t}^{k} w_{x}+\sigma^{\al+1} J\right]_{x}=0,
\end{aligned}
\ee
where

\[
\begin{aligned}
J: &=\partial_{t}^{k-1}\left\{\tilde{\eta}_{x t}\left[\left(\tilde{\eta}_{x}+w_{x}\right)^{-\gamma-1}-\tilde{\eta}_{x}^{-\gamma-1}\right]\right\} \\
&+\left\{\partial_{t}^{k-1}\left[\left(\tilde{\eta}_{x}+w_{x}\right)^{-\gamma-1} w_{x t}\right]-\left(\tilde{\eta}_{x}+w_{x}\right)^{-\gamma-1} \partial_{t}^{k} w_{x}\right\}. \end{aligned}
\]
To obtain the leading terms of $J$, we single out the terms involving $\partial_{t}^{k-1} w_{x}$. To this end, we rewrite $J$ as
\be\label{ej}
\begin{aligned}
J=& \tilde{\eta}_{x t} \partial_{t}^{k-1}\left[\left(\tilde{\eta}_{x}+w_{x}\right)^{-\gamma-1}-\tilde{\eta}_{x}^{-\gamma-1}\right] \\
&+(k-1)\left[\left(\tilde{\eta}_{x}+w_{x}\right)^{-\gamma-1}\right]_{t} \partial_{t}^{k-1} w_{x} \\
&+w_{x t} \partial_{t}^{k-1}\left[\left(\tilde{\eta}_{x}+w_{x}\right)^{-\gamma-1}\right]\\
&+\sum_{\ell=1}^{k-1} C_{k-1}^{\ell}\left(\partial_{t}^{\ell} \tilde{\eta}_{x t}\right) \partial_{t}^{k-1-\ell}\left[\left(\widetilde{\eta}_{x}+w_{x}\right)^{-\gamma-1}-\widetilde{\eta}_{x}^{-\gamma-1}\right] \\
&+\sum_{\ell=2}^{k-2} C_{k-1}^{\ell}\left(\partial_{t}^{k-\ell} w_{x}\right) \partial_{t}^{\ell}\left[\left(\widetilde{\eta}_{x}+w_{x}\right)^{-\gamma-1}\right] \\
&=k\left[\lt(\tilde{\eta}_{x}+w_{x}\right)^{-\gamma-1}\right]_{t} \partial_{t}^{k-1} w_{x}+\widetilde{J},
\eali
\ee
where
\be\label{ej1}
\bali
\tilde{J}:=&-(\gamma+1)\left(\tilde{\eta}_{x}+w_{x}\right)_{t} \sum_{\ell=1}^{k-2} C_{k-2}^{\ell}\left(\partial_{t}^{k-1-\ell} w_{x}\right) \partial_{t}^{\ell}\left[\left(\tilde{\eta}_{x}+w_{x}\right)^{-\gamma-2}\right]\\
&-(\gamma+1)\left\{\tilde{\eta}_{x t} \partial_{t}^{k-2}\left\{\widetilde{\eta}_{x t}\left[\left(\widetilde{\eta}_{x}+w_{x}\right)^{-\gamma-2}-\tilde{\eta}_{x}^{-\gamma-2}\right]\right\}\right. \\
&\qq\qq\q \left.+w_{x t} \partial_{t}^{k-2}\left[\left(\widetilde{\eta}_{x}+w_{x}\right)^{-\gamma-2} \tilde{\eta}_{x t}\right]\right\} \\
&+\sum_{\ell=1}^{k-1} C_{k-1}^{\ell}\left(\partial_{t}^{\ell} \tilde{\eta}_{x t}\right) \partial_{t}^{k-1-\ell}\left[\left(\widetilde{\eta}_{x}+w_{x}\right)^{-\gamma-1}-\widetilde{\eta}_{x}^{-\gamma-1}\right] \\
&+\sum_{\ell=2}^{k-2} C_{k-1}^{\ell}\left(\partial_{t}^{k-\ell} w_{x}\right) \partial_{t}^{\ell}\left[\left(\widetilde{\eta}_{x}+w_{x}\right)^{-\gamma-1}\right].
\eali
\ee
Here summations $\sum_{\ell=1}^{k-2}$ and $\sum_{\ell=2}^{k-2}$ are understood to be 0 when $k=1,2$ and $k=1,2,3,$ respectively. It should be noted that only the terms of lower-order derivatives, $w_{x}, \ldots, \partial_{t}^{k-2} w_{x},$ are contained in $\tilde{J} .$ In particular, $\tilde{J}=0$ when $k=1$.
\begin{lemma}\label{ehweightede}
 Suppose that \eqref{eenergyass} holds for some small positive number $\epsilon_{0} \in(0,1)$. Then for all $j=1, \ldots, m$
\be\label{ehenergy1}
\begin{aligned}
&\mathcal{E}_{j}(t)+\int_{0}^{t} \int\left[(1+\tau)^{2 j+1-\dl{\bf1}_{\la<1}} \sigma^\al\left(\partial_{\tau}^{j+1} w\right)^{2}\rt.\\
&\qq\qq+\lt. (1+\tau)^{2 j-1-\dl{\bf1}_{\la<1}}\sigma^{\al+1}\left(\partial_{\tau}^{j} w_{x}\right)^{2}\rt] d x d\tau\\
\lesssim& \sum_{\ell=0}^{j} \mathcal{E}_{\ell}(0), \quad t \in[0, T].\\
\end{aligned}
\ee
\end{lemma}

\begin{proof} We use induction to prove \eqref{ehenergy1}. As shown in Lemma \ref{lbasice}, we know that \eqref{ehenergy1} holds for $j=0 .$ For $1 \leq k \leq m,$ we make the induction hypothesis that \eqref{ehenergy1} holds for all $j=0,1, \ldots, k-1,$ i.e.,
\be\label{ehenergy2}
\begin{aligned}
&\mathcal{E}_{j}(t)+\int_{0}^{t} \int\left[(1+\tau)^{2 j+1-\dl{\bf1}_{\la<1}} \sigma^\al\left(\partial_{\tau}^{j+1} w\right)^{2}\rt.\\
&\qq\qq+\lt. (1+\tau)^{2 j-1-\dl{\bf1}_{\la<1}}\sigma^{\al+1}\left(\partial_{\tau}^{j} w_{x}\right)^{2}\rt] d x d\tau\\
\lesssim& \sum_{\ell=0}^{j} \mathcal{E}_{\ell}(0), \quad t \in[0, T], \quad 0\leq j\leq k-1.\\
\end{aligned}
\ee

It suffices to prove \eqref{ehenergy1} holds for $j=k$ under the induction hypothesis \eqref{ehenergy2}.
We divide the proof into three steps.

\noindent{\bf Step one: Setup of the linearized main term}

We begin by rewriting \eqref{ehweighted0} as follows

\be\label{ehweighted0r}
\begin{aligned}
&\sigma^\al \partial_{t}^{k+2} w-\left[\sigma^{\al+1}\t{\eta}^{-\g-1}_x(1+o(1)) \partial_{t}^{k} w_{x}\rt]_x+ \f{\mu}{(1+t)^\la}\sigma^\al \partial_{t}^{k+1} w\\
=& -\mu\sigma^\al\sum^k_{\ell=1}C^\ell_k\p^\ell_t(1+t)^{-\la}\p^{k+1-\ell}_t w+\lt[\sigma^{\al+1} J\right]_{x}\\
:=& P(x,t).
\end{aligned}
\ee
If we view $\p^k_t w$ as $w$ in the proof of Lemma \ref{lbasice}, we can get a similar formula with \eqref{ebasic4r5} and \eqref{ebasic3} as follows
\bes
\bali
&\f{d}{dt}\int{\mathfrak{E}}_k(t)dx+(1+t)^{1-\dl{\bf1}_{\la<1}}\int \sigma^\al (\p^{k+1}_tw)^{2} d x\\
&+(1+t)^{-1-\dl{\bf1}_{\la<1}}\int \sigma^{\al+1}(\p^{k}_tw_x)^{2} d x\\
&\ls (1+t)^{1+\la-\dl{\bf1}_{\la<1}}\int P(x,t) \p^{k+1}_t wdx\\
&+ (1+t)^{\la-\dl{\bf1}_{\la<1}}\int P(x,t) \p^{k}_t wdx,
\eali
\ees
where
\bes
\bali
{\mathfrak{E}}_k(x,t)\thickapprox&  (1+t)^{1+\la-\dl{\bf1}_{\la<1}}\sigma^\al (\p^{k+1}_tw)^{2}\\
&+(1+t)^{-\dl{\bf1}_{\la<1}}\lt[\sigma^{\al+1}(\p^{k}_tw_x)^2 + \sigma^\al (\p^{k}_tw)^2\rt],
\eali
\ees
and
\bes
(1+t)^{2 k}\int{\mathfrak{E}}_k(x,t)dx\thickapprox \mathcal{E}_k(t).
\ees

Then by using integration by parts, and \eqref{ehweighted0r}, we can get \be\label{ehweighted6r}
\bali
&\f{d}{dt}\int{\mathfrak{E}}_k(t)dx\\
&+ \int \lt[(1+t)^{1-\dl{\bf1}_{\la<1}} \sigma^\al (\p^{k+1}_tw)^{2} +(1+t)^{-1-\dl{\bf1}_{\la<1}}\sigma^{\al+1} (\p^{k}_tw_x)^{2}\rt] d x\\
&\ls - (1+t)^{1+\la-\dl{\bf1}_{\la<1}}\int \sigma^{\al+1}J \p^{k+1}_t w_xd x\\
&- (1+t)^{\la-\dl{\bf1}_{\la<1}}\int \sigma^{\al+1}J \p^{k}_t w_xd x\\
&-\sum^k_{\ell=1}\int(1+t)^{1-\ell-\dl{\bf1}_{\la<1}}\sigma^\al\p^{k+1-\ell}_t w \partial_{t}^{k+1} wdx\\
&-\sum^k_{\ell=1}\int(1+t)^{-\ell-\dl{\bf1}_{\la<1}}\sigma^\al\p^{k+1-\ell}_t w \partial_{t}^{k} wdx.
\eali
\ee

Since the derivative of the term containing $\p^{j+1}_tw_x$ on the right hand of \eqref{ehweighted6r} exceeds the highest order derivative on the left side of \eqref{ehweighted6r}, we use \eqref{ej} and integration by parts on time to estimate the first term on the right-hand side of \eqref{ehweighted6r} as follows:

\be\label{ehweighted3}
\bali
&-(1+t)^{1+\la-\dl{\bf1}_{\la<1}}\int \sigma^\al J \partial_{t}^{k+1} w_{x} d x \\
=&\frac{d}{d t} \int (1+t)^{1+\la-\dl{\bf1}_{\la<1}}\sigma^{\al+1} J \partial_{t}^{k} w_{x} d x\\
&+(1+\la-\dl{\bf1}_{\la<1})(1+t)^{\la-\dl{\bf1}_{\la<1}}\int\sigma^{\al+1}J \partial_{t}^{k} w_{x} d x\\
 &+(1+t)^{1+\la-\dl{\bf1}_{\la<1}}\int \sigma^{\al+1}J_{t} \partial_{t}^{k} w_{x} d x\\
\quad=&-\frac{d}{d t} \int \sigma^{\al+1}(1+t)^{1+\la-\dl{\bf1}_{\la<1}}J \partial_{t}^{k} w_{x} d x\\
&+(1+\la-\dl{\bf1}_{\la<1})(1+t)^{\la-\dl{\bf1}_{\la<1}}\int \sigma^{\al+1} J \partial_{t}^{k} w_{x} d x\\
 &+k (1+t)^{1+\la-\dl{\bf1}_{\la<1}}\int \sigma^{\al+1}\left[\left(\t{\eta}_x+w_{x}\right)^{-\gamma-1}\right]_{t}\left(\partial_{t}^{k} w_{x}\right)^{2} d x \\
&+k(1+t)^{1+\la-\dl{\bf1}_{\la<1}}\int \sigma^{\al+1}\left[\left(\t{\eta}_x+w_{x}\right)^{-\gamma-1}\right]_{t t} \partial_{t}^{k-1} w_{x} \partial_{t}^{k} w_{x} d x\\
&+(1+t)^{1+\la-\dl{\bf1}_{\la<1}}\int \sigma^{\al+1}\widetilde{J}_{t} \partial_{t}^{k} w_{x} d x.
\eali
\ee
Inserting \eqref{ehweighted3} into \eqref{ehweighted6r}, we can get

\be\label{ehweighted6}
\bali
&\f{d}{dt}\int \lt[{\mathfrak{E}}_k(t)+\sigma^{\al+1} (1+t)^{1+\la-\dl{\bf1}_{\la<1}}J \partial_{t}^{k} w_{x}\rt]dx\\
 &\q +\int\lt[  (1+t)^{1-\dl{\bf1}_{\la<1}}\sigma^\al (\p^{k+1}_tw)^{2} + (1+t)^{-1-\dl{\bf1}_{\la<1}}\sigma^{\al+1} (\p^{k}_tw_x)^{2}\rt] d x\\
\ls& \sum^k_{\ell=1}(1+t)^{1-\ell-\dl{\bf1}_{\la<1}}\lt|\int\sigma^\al\p^{k+1-\ell}_t w \partial_{t}^{k+1} wdx\rt|\\
&+\sum^k_{\ell=1}(1+t)^{-\ell-\dl{\bf1}_{\la<1}}\lt|\int\sigma^\al\p^{k+1-\ell}_t w \partial_{t}^{k} wdx\rt|\\
 &+(1+t)^{1+\la-\dl{\bf1}_{\la<1}}\int \sigma^{\al+1}\left[\left(\t{\eta}_x+w_{x}\right)^{-\gamma-1}\right]_{t}\left(\partial_{t}^{k} w_{x}\right)^{2} d x \\
&+(1+t)^{1+\la-\dl{\bf1}_{\la<1}}\lt|\int \sigma^{\al+1}\left[\left(\t{\eta}_x+w_{x}\right)^{-\gamma-1}\right]_{t t} \partial_{t}^{k-1} w_{x} \partial_{t}^{k} w_{x} d x\rt|\\
&+(1+t)^{\la-\dl{\bf1}_{\la<1}}\lt|\int \sigma^{\al+1} J \partial_{t}^{k} w_{x} d x\rt|\\
&+(1+t)^{1+\la-\dl{\bf1}_{\la<1}}\lt|\int \sigma^{\al+1}\widetilde{J}_{t} \partial_{t}^{k} w_{x} d x\rt|\\
&:=\sum^6_{i=1} I_{i}.
\eali
\ee

Here, we don't take absolute value of term $I_3$ since we have the highest order derivative $\partial_{t}^{k} w_{x}$ and there is no smallness on $\t{\eta}_{x}$. We need to estimate it by combining the smallness of $w_{xt}$ and the sign of $\t{\eta}_{xt}$. Now we estimate $I_i\ (1\leq i\leq 6)$ term by term by using \eqref{ehenergy2} and the left of \eqref{ehweighted6}. Next we only consider the case for $\la<1$ and the proof below is still valid for $\la=1$ by simply replacing $\la=1$ and $\dl=0$. Also for simplicity of presentation and notations, sometimes we denote $\beta:=\la+1$.

\noindent {\bf Step two: Estimates of the low order and nonlinear terms}

For $I_1$, by using Cauchy-Schwartz inequality, we have, for a small constant $\nu$,
\bes
\bali
I_1\ls& \nu (1+t)^{1-\dl}\int\sigma^\al (\p^{k+1}_tw)^2dx\\
&+C_{\nu}\sum^{k}_{\ell=1} (1+t)^{1-2\ell-\dl}\int\sigma^\al (\p^{k+1-\ell}_tw)^2dx\\
\ls& \nu (1+t)^{1-\dl}\int\sigma^\al (\p^{k+1}_tw)^2dx\\
& +(1+t)^{-2k}C_{\nu}\sum^{k}_{\ell=1} (1+t)^{2(k+1-\ell)-1-\dl}\int\sigma^\al (\p^{k+1-\ell}_tw)^2dx\\
\ls& \nu (1+t)^{1-\dl}\int\sigma^\al (\p^{k+1}_tw)^2dx\\
& +(1+t)^{-2k}C_{\nu}\sum^{k}_{\ell=1} (1+t)^{2\ell-1-\dl}\int\sigma^\al (\p^{\ell}_tw)^2dx.
\eali
\ees
The first term of the above inequality can be absorbed by the left hand of \eqref{ehweighted6} and the time integral of the second term can be bounded by the initial data from the induction assumption \eqref{ehenergy2}.

For $I_2$, we can estimate it similarly with $I_1$ as follows
\bes
\bali
|I_2|\ls& (1+t)^{-1-\dl}\int\sigma^\al (\p^{k}_tw)^2dx\\
&+\sum^{k}_{\ell=1} (1+t)^{1-2\ell-\dl}\int\sigma^\al (\p^{k+1-\ell}_tw)^2dx\\
\ls&(1+t)^{-2k}\sum^{k}_{\ell=1} (1+t)^{2\ell-1-\dl}\int\sigma^\al (\p^{\ell}_tw)^2dx.
\eali
\ees
For term $I_3$, we first have
\be
\bali
I_3&:=(1+t)^{1+\la-\dl}\int\sigma^{\al+1}\left[\left(\t{\eta}_x+w_{x}\right)^{-\gamma-1}\right]_{t}\left(\partial_{t}^{k} w_{x}\right)^{2} d x\\
&=(-\g-1) (1+t)^{1+\la-\dl}\int\sigma^{\al+1}\left(\t{\eta}_x+w_{x}\right)^{-\gamma-2}(\t{\eta}_{xt}+w_{xt})\left(\partial_{t}^{k} w_{x}\right)^{2} d x\\
&\ls  (1+t)^{1+\la-\dl}\lt|\int\sigma^{\al+1}\left(\t{\eta}_x+w_{x}\right)^{-\gamma-2}w_{xt}\left(\partial_{t}^{k} w_{x}\right)^{2} d x\rt|.
\eali
\ee
Here we just throw away the term containing $\t{\eta}_{xt}$ due to its nonnegative property. Then we use \eqref{esobolev}, \eqref{ecor3}, and \eqref{ecor4} to obtain that

\be
\bali
I_3&\ls \|(1+t)^{2+\la}\left(\t{\eta}_x+w_{x}\right)^{-\gamma-2}\t{w}_{xt}\|_{L^\i} (1+t)^{-1-\dl}\lt| \int \sigma^{\al+1}\left(\partial_{t}^{k} w_{x}\right)^{2} d x\rt|\\
&\ls  \s{\e_0}(1+t)^{\f{\dl}{2}-\f{\la+1}{\g+1}} (1+t)^{-1-\dl}\lt|\int  \sigma^{\al+1}\left(\partial_{t}^{k} w_{x}\right)^{2} d x\rt|\\
&\ls \s{\e_0} (1+t)^{-1-\dl}\int  \sigma^{\al+1}\left(\partial_{t}^{k} w_{x}\right)^{2} d x,
\eali
\ee
which can be absorbed by the positive term on the lefthand of \eqref{ehweighted6} if $\e_0$ is small enough.

For term $I_4$, using \eqref{ecor3}, \eqref{ecor4} and \eqref{esobolev}, it is easy to deduce that
\bes
\left[\left(\t{\eta}_x+w_{x}\right)^{-\gamma-1}\right]_{t t}\ls (1+t)^{-2-\beta}+(1+t)^{-\beta-\f{\beta}{\g+1}}|w_{xtt}|.
\ees
Then using Cauchy-Schwartz inequality, we have
\bes
\bali
|I_4|&:=(1+t)^{1+\la-\dl}\lt|\int  \sigma^{\al+1}\left[\left(\t{\eta}_x+w_{x}\right)^{-\gamma-1}\right]_{t t} \partial_{t}^{k-1} w_{x} \partial_{t}^{k} w_{x} d x\rt|\\
&\ls (1+t)^{1+\la-\dl}\int \sigma^{\al+1}\left((1+t)^{-2-\beta}\rt.\\
&\qq\qq\qq\qq\lt.+(1+t)^{-\beta-\f{\beta}{\g+1}}|w_{xtt}|\right) |\partial_{t}^{k-1} w_{x}| | \partial_{t}^{k} w_{x}| d x\\
&\ls \nu(1+t)^{-1-\dl} \int  \sigma^{\al+1} \lt(\partial_{t}^{k} w_{x}\rt)^2 d x\\
&+C_\nu (1+t)^{-3-\dl} \int  \sigma^{\al+1} \lt(\partial_{t}^{k-1} w_{x}\rt)^2 d x \\
 &+ C_\nu(1+t)^{1-\f{2\beta}{\g+1}-\dl} \|\sigma^{1/2}w_{xtt}\|^2_{L^\i}\int\sigma^{\al} |\partial_{t}^{k-1} w_{x}| ^2 d x\\
 &\ls \nu(1+t)^{-1-\dl} \int  \sigma^{\al+1} \lt(\partial_{t}^{k} w_{x}\rt)^2 d x\\
 &+C_\nu (1+t)^{-2k} (1+t)^{2(k-1)-1-\dl}\int  \sigma^{\al+1} \lt(\partial_{t}^{k-1} w_{x}\rt)^2 d x\\
 &+C_\nu (1+t)^{-2k-1-\f{2\beta}{\g+1}+\dl}\int\sigma^{\al-1} |\partial_{t}^{k-1} w_{x}| ^2 d x,
 \eali
 \ees
 where at the last line we have used the boundedness of $\sigma$. From the definition of $\mathcal{E}_{j,i}$, we get
 \bes
 \bali
 |I_4|&\ls \nu(1+t)^{-1-\dl} \int  \sigma^{\al+1} \lt(\partial_{t}^{k} w_{x}\rt)^2 d x\\
 &+C_\nu (1+t)^{-2k} (1+t)^{2(k-1)-1-\dl}\int  \sigma^{\al+1} \lt(\partial_{t}^{k-1} w_{x}\rt)^2 d x\\
 &+ C_\nu (1+t)^{-2k-1^+}\mathcal{E}_{k-1,0}.
\eali
\ees
where $1^+$ is a constant bigger than $1$.

Now we come to estimate the terms involving $J$ and $\t{J}$, which is a little complicated.

It needs to bound $J,$ which contains lower-order terms involving $w_{x}, \ldots, \partial_{t}^{k-1} w_{x}$. Following from \eqref{ecor3}, \eqref{ecor4}, \eqref{esobolev} and \eqref{ej}, one has
\be\label{ehweighted3rr}
\begin{aligned}
J^{2} & \lesssim\left|\left[\left(\tilde{\eta}_{x}+w_{x}\right)^{-\gamma-1}\right]_{t}\right|^{2}\left(\partial_{t}^{k-1} w_{x}\right)^{2}+\tilde{J}^{2} \\
& \lesssim(1+t)^{-2\beta-2}\left(\partial_{t}^{k-1} w_{x}\right)^{2}+\widetilde{J}^{2}.
\end{aligned}
\ee

In view of \eqref{ecor3}, \eqref{ecor4}, \eqref{esobolev} and \eqref{ej1}, we have
\be\label{ehweighted2r}
\begin{aligned}
|\tilde{J}| \lesssim & \sum_{\ell=1}^{k-2}(1+t)^{-\beta-1-\ell}\left|\partial_{t}^{k-1-\ell} w_{x}\right|+(1+t)^{-\beta-\frac{\beta}{\gamma+1}}\sum^{k-2}_{\ell=2}\left|\partial_{t}^{\ell} w_{x}\right|\left|\partial_{t}^{k-\ell} w_{x}\right| \\
&+(1+t)^{-\beta-1-\frac{\beta}{\gamma+1}}\sum^{k-3}_{\ell=2}\left|\partial_{t}^{\ell} w_{x}\right|\left|\partial_{t}^{k-1-\ell} w_{x}\right| \\
&+(1+t)^{-\beta-2-\frac{\beta}{\gamma+1}}\sum^{k-4}_{\ell=2}\left|\partial_{t}^{\ell} w_{x}\right|\left|\partial_{t}^{k-2-\ell} w_{x}\right|+l.o.t..
\end{aligned}
\ee
Here and thereafter the notation $l .o . t.$ is used to represent the lower-order terms involving $\partial^{\ell}_{t} w_{x}$ with $\ell=2, \ldots, k-2 .$ It should be noticed that the second term on the right-hand side of \eqref{ehweighted2r} only appears as $k-2 \geq 2,$ the third term as $k-3 \geq 2$ and the fourth term as $k-4\geq 2$.


From \eqref{esobolev}, we have

\be\label{eetj}
\begin{aligned}
|\tilde{J}| \lesssim & \sum_{\ell=1}^{k-2}(1+t)^{-\beta-1-\ell}\left|\partial_{t}^{k-1-\ell} w_{x}\right|\\
 &+\s{\e_0}(1+t)^{-\beta-\frac{\beta}{\gamma+1}}\sum^{[\f{k-2}{2}]}_{\ell=2}\sigma^{-\f{\ell-1}{2}}(1+t)^{-\ell}\left|\partial_{t}^{k-\ell} w_{x}\right| \\
&+\s{\e_0}(1+t)^{-\beta-1-\frac{\beta}{\gamma+1}}\sum^{[\f{k-3}{2}]}_{\ell=2}\sigma^{-\f{\ell-1}{2}}(1+t)^{-\ell}\left|\partial_{t}^{k-1-\ell} w_{x}\right| \\
&+\s{\e_0}(1+t)^{-\beta-2-\frac{\beta}{\gamma+1}}\sum^{[\f{k-4}{2}]}_{\ell=2}\sigma^{-\f{\ell-1}{2}}(1+t)^{-\ell}\left|\partial_{t}^{k-2-\ell} w_{x}\right|.
\eali
\ee
So, using Cauchy-Schwartz inequality, we estimate $I_5$ as follows
\bes
\bali
|I_5|&\ls\lt|(1+t)^{\la-\dl}\int \sigma^{\al+1} J \partial_{t}^{k} w_{x} d x\rt|\\
&\ls \nu(1+t)^{-1-\dl}\int \sigma^{\al+1}  \lt(\partial_{t}^{k} w_{x}\rt)^2 d x+C_\nu (1+t)^{2\la+1-\dl}\int \sigma^{\al+1} J^2 d x.
\eali
\ees
While, from \eqref{ehweighted3rr},
\be\label{ej2}
\bali
&(1+t)^{2\la+1-\dl}\int \sigma^{\al+1} J^2 d x\\
\ls& (1+t)^{2\la+1-\dl}\int \sigma^{\al+1} \lt[(1+t)^{-2\beta-2}\left(\partial_{t}^{k-1} w_{x}\right)^{2}+\widetilde{J}^{2}\rt] d x\\
\ls&(1+t)^{-3-\dl}\int \sigma^{\al+1} \left(\partial_{t}^{k-1} w_{x}\right)^{2}dx+(1+t)^{2\la+1-\dl}\int \sigma^{\al+1} \widetilde{J}^{2} d x.
\eali
\ee
This implies that
\be\label{ehweighted7}
\bali
|I_5|\ls& \nu(1+t)^{-1-\dl}\int \sigma^{\al+1}  \lt(\partial_{t}^{k} w_{x}\rt)^2 d x\\
&+C_\nu (1+t)^{-3-\dl}\int \sigma^{\al+1}  \lt(\partial_{t}^{k-1} w_{x}\rt)^2 d x\\
  & +(1+t)^{2\la+1-\dl}\int \sigma^{\al+1} \t{J}^{2} d x.
\eali
\ee
Next we estimate the term involving $\t{J}$ as follows by using \eqref{eetj}
\be
\bali
 &(1+t)^{2\la+1-\dl}\int \sigma^{\al+1} \t{J}^{2} d x\\
\ls & \sum_{\ell=1}^{k-2}(1+t)^{-3-2\ell-\dl}\int \sigma^{\al+1}\left|\partial_{t}^{k-1-\ell} w_{x}\right|dx\\
 &+\e_0(1+t)^{-1-\frac{2\beta}{\gamma+1}-\dl}\sum^{[\f{k-2}{2}]}_{\ell=2}\int\sigma^{\al-\ell+2}(1+t)^{-2\ell}\left(\partial_{t}^{k-\ell} w_{x}\right)^2dx \\
&+\e_0(1+t)^{-3-\frac{2\beta}{\gamma+1}-\dl}\sum^{[\f{k-3}{2}]}_{\ell=2}\int\sigma^{\al-\ell+2}(1+t)^{-2\ell}\left(\partial_{t}^{k-1-\ell} w_{x}\right)^2dx \\
&+\e_0(1+t)^{-5-\frac{2\beta}{\gamma+1}-\dl}\sum^{[\f{k-4}{2}]}_{\ell=2}\int\sigma^{\al-\ell+2}(1+t)^{-2\ell}\left(\partial_{t}^{k-2-\ell} w_{x}\right)dx\\
\ls &(1+t)^{-2k}\sum_{\ell=1}^{k-2}(1+t)^{2(k-1-\ell)-1-\dl}\int \sigma^{\al+1}\left|\partial_{t}^{k-1-\ell} w_{x}\right|dx\\
 &+\e_0(1+t)^{-1-\frac{2\beta}{\gamma+1}-\dl}\sum^{[\f{k}{2}]}_{\ell=2}\int\sigma^{\al-\ell+2}(1+t)^{-2\ell}\left(\partial_{t}^{k-\ell} w_{x}\right)^2dx, \\
\eali
\ee
where at the fourth and fifth line of the above inequality, we have viewed $\ell+1$ and $\ell+2$ to be the new $\ell$ and used the fact that $|\sigma|\leq A$.
In view of the Hardy inequality \eqref{ehardy}, we see that for $\ell=2, \ldots,[k / 2]$, $\al+2-\ell>-1$, then we have
\bes
\begin{aligned}
\int \sigma^{\alpha+2-\ell}\left|\partial_{t}^{k-\ell} w_{x}\right|^{2} d x \lesssim &\int \sigma^{\alpha+2-\ell+2}\sum^{2}_{n=1}\left(\partial_{t}^{k-\ell} \p^{n}_x w\right)^{2} d x \\
 &\underbrace{\cdots\cdots}_{\ell-2\ times}\\
 \lesssim&  \int \sigma^{\alpha+\ell}\sum^{\ell}_{n=1}\left(\partial_{t}^{k-\ell} \p^{n}_x w\right)^{2} d x \\
 \lesssim&  \sum^{\ell}_{n=1}\int \sigma^{\alpha+n}\left(\partial_{t}^{k-\ell} \p^{n}_x w\right)^{2} d x\\
 \lesssim & \sum^{\ell}_{n=1}(1+t)^{-2 (k-\ell)+\dl} \mathcal{E}_{k-\ell, n-1}.
\end{aligned}
\ees
The above two inequalities indicate that
\bes
\bali
&(1+t)^{2\la+1-\dl}\int \sigma^{\al+1} \t{J}^2 d x\\
\ls&  (1+t)^{-2k}\sum_{\ell=1}^{k-2}(1+t)^{2\ell-1-\dl}\int \sigma^{\al+1}\left|\partial_{t}^{\ell} w_{x}\right|dx\\
 &+\e_0(1+t)^{-2k-1-\frac{2\beta}{\gamma+1}}\sum^{[\f{k}{2}]}_{\ell=2} \sum^{\ell}_{n=1} \mathcal{E}_{k-\ell, n-1}\\
\ls& (1+t)^{-2k}\sum_{\ell=1}^{k-2}(1+t)^{2\ell-1-\dl}\int \sigma^{\al+1}\left(\partial_{t}^{\ell} w_{x}\right)^2dx\\
 &+\e_0(1+t)^{-2k-1-\frac{2\beta}{\gamma+1}}\sum^{k-1}_{\ell=0} \mathcal{E}_{\ell},
\eali
\ees
where we have used Proposition \ref{pelliptic}.
Combining the above inequality and \eqref{ehweighted7}, we can get
\bes
\bali
|I_5|&\ls \nu(1+t)^{-1-\dl}\int\sigma^{\al+1}  \lt(\partial_{t}^{k} w_{x}\rt)^2 d x\\
&+(1+t)^{-2k}\sum_{\ell=1}^{k-1}(1+t)^{2\ell-1-\dl}\int \sigma^{\al+1}\left(\partial_{t}^{\ell} w_{x}\right)^2dx\\
 &+\e_0(1+t)^{-2k-1-\frac{2\beta}{\gamma+1}} \sum^{k-1}_{\ell=0}\mathcal{E}_{\ell}.
\eali
\ees
Next we come to deal with the term $I_6$ involving $\t{J}_t$. The same estimate with that for $\t{J}$, We can have
\bes
\begin{aligned}
|\tilde{J}_t| \lesssim & \sum_{\ell=1}^{k}(1+t)^{-\beta-1-\ell}\left|\partial_{t}^{k-\ell} w_{x}\right|+(1+t)^{-\beta-\frac{\beta}{\gamma+1}}\sum^{k-1}_{\ell=2}\left|\partial_{t}^{\ell} w_{x}\right|\left|\partial_{t}^{k+1-\ell} w_{x}\right| \\
&+(1+t)^{-\beta-1-\frac{\beta}{\gamma+1}}\sum^{k-2}_{\ell=2}\left|\partial_{t}^{\ell} w_{x}\right|\left|\partial_{t}^{k-\ell} w_{x}\right| \\
&+(1+t)^{-\beta-2-\frac{\beta}{\gamma+1}}\sum^{k-3}_{\ell=2}\left|\partial_{t}^{\ell} w_{x}\right|\left|\partial_{t}^{k-1-\ell} w_{x}\right|+l.o.t.,
\end{aligned}
\ees
which implies, by using \eqref{esobolev}, that
\bes
\begin{aligned}
|\tilde{J}_t| \lesssim & \sum_{\ell=1}^{k}(1+t)^{-\beta-1-\ell}\left|\partial_{t}^{k-\ell} w_{x}\right|\\
 &+\s{\e_0}(1+t)^{-\beta-\frac{\beta}{\gamma+1}}\sum^{[\f{k-1}{2}]}_{\ell=2}\sigma^{-\f{\ell-1}{2}}(1+t)^{-\ell}\left|\partial_{t}^{k+1-\ell} w_{x}\right| \\
&+\s{\e_0}(1+t)^{-\beta-1-\frac{\beta}{\gamma+1}}\sum^{[\f{k-2}{2}]}_{\ell=2}\sigma^{-\f{\ell-1}{2}}(1+t)^{-\ell}\left|\partial_{t}^{k-\ell} w_{x}\right| \\
&+\s{\e_0}(1+t)^{-\beta-2-\frac{\beta}{\gamma+1}}\sum^{[\f{k-3}{2}]}_{\ell=2}\sigma^{-\f{\ell-1}{2}}(1+t)^{-\ell}\left|\partial_{t}^{k-1-\ell} w_{x}\right|\\
&\ls\sum_{\ell=1}^{k}(1+t)^{-\beta-1-\ell}\left|\partial_{t}^{k-\ell} w_{x}\right|\\
 &+\s{\e_0}(1+t)^{-\beta-\frac{\beta}{\gamma+1}}\sum^{[\f{k-1}{2}]}_{\ell=2}\sigma^{-\f{\ell-1}{2}}(1+t)^{-\ell}\left|\partial_{t}^{k+1-\ell} w_{x}\right| \\
 &+\s{\e_0}(1+t)^{-\beta-1-\frac{\beta}{\gamma+1}}\sum^{[\f{k-1}{2}]}_{\ell=2}\sigma^{-\f{\ell-1}{2}}(1+t)^{-\ell}\left|\partial_{t}^{k-\ell} w_{x}\right|.
\end{aligned}
\ees
At the fourth line of the above inequality, we have viewed $\ell+1$ to be the new $\ell$. Then we estimates $I_6$ as follows,
\be\label{etj3}
\bali
|I_6|&:=\lt|\int \sigma^{\al+1}(1+t)^{1+\la-\dl}\t{J}_t \partial_{t}^{k} w_{x} d x\rt|\\
&\ls \nu\int \sigma^{\al+1}(1+t)^{-1-\dl}\lt( \partial_{t}^{k} w_{x}\rt)^2 d x+C_\nu (1+t)^{2\la+3-\dl}\int \sigma^{\al+1}|\t{J}_t|^2  d x.
\eali
\ee
And
\be\label{etj}
\bali
&(1+t)^{2\la+3-\dl}\int \sigma^{\al+1}|\t{J}_t|^2  d x\\
\ls&\sum_{\ell=1}^{k}(1+t)^{-1-2\ell-\dl}\int \sigma^{\al+1}\left(\partial_{t}^{k-\ell} w_{x}\right)^2dx\\
 &+\e_0(1+t)^{1-\frac{2\beta}{\gamma+1}-\dl} \sum^{[\f{k-1}{2}]}_{\ell=2}\int\sigma^{\al+2-\ell}(1+t)^{-2\ell}\left(\partial_{t}^{k+1-\ell} w_{x}\right)^2dx \\
 &+\e_0(1+t)^{-1-\frac{2\beta}{\gamma+1}-\dl}\sum^{[\f{k-1}{2}]}_{\ell=2}\int\sigma^{\al+2-\ell}(1+t)^{-2\ell}\left(\partial_{t}^{k-\ell} w_{x}\right)^2dx.
 \eali
\ee
Again using \eqref{ehardy} repeatedly, we can have

\bes
\bali
&\int\sigma^{\al+2-\ell}(1+t)^{-2\ell}\left(\partial_{t}^{k+1-\ell} w_{x}\right)^2dx\\
 \ls &\sum^{\ell}_{n=1}(1+t)^{-2 (k+1-\ell)+\dl} \mathcal{E}_{k+1-\ell, n-1}, \eali
\ees
and
\bes
\bali
&\int\sigma^{\al+2-\ell}(1+t)^{-2\ell}\left(\partial_{t}^{k+1-\ell} w_{x}\right)^2dx\\
&+\int\sigma^{\al+2-\ell}(1+t)^{-2\ell}\left(\partial_{t}^{k-\ell} w_{x}\right)^2dx\\
 \ls &\sum^{\ell}_{n=1}(1+t)^{-2 (k-\ell)+\dl} \mathcal{E}_{k-\ell, n-1}.
 \eali
\ees
Inserting the above two inequalities into \eqref{etj}, we have

\be\label{etj1}
\bali
&(1+t)^{2\la+3-\dl}\int \sigma^{\al+1}|\t{J}_t|^2  d x\\
 \ls&(1+t)^{-2k}\sum_{\ell=0}^{k-1}(1+t)^{2\ell-1-\dl}\int \sigma^{\al+1}\left(\partial_{t}^\ell w_{x}\right)^2dx\\
 &+\e_0(1+t)^{-2k-1-\frac{2\beta}{\gamma+1}}\sum^{[\f{k-1}{2}]}_{\ell=2}\sum^\ell_{n=1}\lt(\mathcal{E}_{k+1-\ell,n-1}+\mathcal{E}_{k-\ell,n-1}\rt)\\
 \ls&(1+t)^{-2k}\sum_{\ell=1}^{k-1}(1+t)^{2\ell-1-\dl}\int \sigma^{\al+1}\left(\partial_{t}^{\ell} w_{x}\right)^2dx\\
 &+\e_0(1+t)^{-2k-1-\frac{2\beta}{\gamma+1}}\sum^k_{\ell=0}\mathcal{E}_{\ell}.
 \eali
\ee

Since there appears $\mathcal{E}_k$ on the right hand of the above inequality, which can not either be absorbed by the left hand of \eqref{ehweighted6} or be controlled by the induction assumption \eqref{ehenergy2}. We calculate it further more as follows.

\be\label{etj2}
\bali
&(1+t)^{-2k-1-\frac{2\beta}{\gamma+1}}\mathcal{E}_{k}\\
=&(1+t)^{-1-\frac{2\beta}{\gamma+1}-\dl}\int \lt[(1+t)^\beta\sigma^\al \p^{k+1}_t w+\sigma^{\al+1} (\p^{k}_t w_x)^2\rt]dx\\
 &+(1+t)^{-1-\frac{2\beta}{\gamma+1}-\dl}\int \sigma^{\al} (\p^{k}_t w)^2dx\\
\ls&\int \lt[(1+t)^{1-\dl}\sigma^\al \p^{k+1}_t w+(1+t)^{-1-\dl}\sigma^{\al+1} (\p^{k}_t w_x)^2\rt]dx\\
&+(1+t)^{-2k}(1+t)^{2(k-1)+1-\dl}\int \sigma^{\al} (\p^{k}_t w)^2dx.
\eali
\ee
Then combining \eqref{etj1}, \eqref{etj2} and \eqref{etj3}, we obtain
\bes
\bali
|I_6|\ls& (\nu+\e_0)\int\lt[(1+t)^{1-\dl}\sigma^\al \p^{k+1}_t w+(1+t)^{-1-\dl}\sigma^{\al+1} (\p^{k}_t w_x)^2\rt]dx\\
&+\e_0(1+t)^{-2k-1-\frac{2\beta}{\gamma+1}}\sum^{k-1}_{\ell=0}\mathcal{E}_{\ell}\\
&+(1+t)^{-2k}\sum_{\ell=0}^{k-1}\int(1+t)^{2\ell-1-\dl} \sigma^{\al+1}\lt(\partial_{t}^{\ell} w_{x}\right)^2dx\\
&+(1+t)^{-2k}\sum_{\ell=0}^{k-1}\int(1+t)^{2\ell+1-\dl} \sigma^{\al} \left(\partial_{t}^{\ell+1} w\right)^2dx. \eali
\ees

\noindent{\bf Step three: Finishing proof of Lemma \ref{ehweightede}}

From all the above estimates for terms $I_1$ to $I_6$, we get that

\bes
\bali
&\f{d}{dt}\int \lt[{\mathfrak{E}}_k(t)+\sigma^{\al+1} (1+t)^{1+\la-\dl}J \partial_{t}^{k} w_{x}\rt]dx\\
 &+ \int\lt[ (1+t)^{1-\dl}\sigma^\al (\p^{k+1}_tw)^{2} +(1+t)^{-1-\dl}\sigma^{\al+1}(\p^{k}_tw_x)^{2}\rt] d x\\
\ls& (\nu+\e_0)\int\lt[ (1+t)^{1-\dl}\sigma^\al (\p^{k+1}_tw)^{2} +(1+t)^{-1-\dl}\sigma^{\al+1}(\p^{k}_tw_x)^{2}\rt] d x\\
&+ (1+t)^{-2 k-1^+}\sum^{k-1}_{\ell=0}\mathcal{E}_{\ell}+(1+t)^{-2 k}\sum^{k-1}_{\ell=0}\int\lt[ (1+t)^{2\ell+1-\dl}\sigma^\al (\p^{\ell+1}_tw)^{2}\rt.\\
 &\lt.\qq\qq\q +(1+t)^{2\ell-1-\dl}\sigma^{\al+1}(\p^{\ell}_tw_x)^{2}\rt] d x,
\eali
\ees
where $1^+$ is some constant bigger than $1$.
Then we get by choosing small $\nu$, for some large $N$, to be determined later,
\bes
\bali
&\f{d}{dt}\int \lt[{\mathfrak{E}}_k(t)+\sigma^{\al+1} (1+t)^{1+\la-\dl}J \partial_{t}^{k} w_{x}\rt]dx\\
 &+ N\int\lt[ (1+t)^{1-\dl}\sigma^\al (\p^{k+1}_tw)^{2} +(1+t)^{-1-\dl}\sigma^{\al+1}(\p^{k}_tw_x)^{2}\rt] d x\\
\ls&  (1+t)^{-2 k-1^+}\sum^{k-1}_{\ell=0}\mathcal{E}_{\ell}+(1+t)^{-2 k}\sum^{k-1}_{\ell=0}\int\lt[ (1+t)^{2\ell+1-\dl}\sigma^\al (\p^{\ell+1}_tw)^{2}\rt.\\
 &\lt.\qq\qq\q +(1+t)^{2\ell-1-\dl}\sigma^{\al+1}(\p^{\ell}_tw_x)^{2}\rt] d x.
\eali
\ees
Multiplying the above inequality by $(1+t)^{2 k}$, we can get
\be\label{ehweighted10}
\bali
&\f{d}{dt}\lt\{(1+t)^{2 k}\int\lt[{\mathfrak{E}}_k(t)+\sigma^{\al+1} (1+t)^{1+\la-\dl}J \partial_{t}^{k} w_{x}\rt]dx\rt\}\\
& -2 k(1+t)^{2 k-1}\int\lt[{\mathfrak{E}}_k(t)+\sigma^{\al+1} (1+t)^{1+\la-\dl}J \partial_{t}^{k} w_{x}\rt]dx\\
&+ N\int\lt[ \sigma^\al (1+t)^{2 k+1-\dl}(\p^{k+1}_tw)^{2} +(1+t)^{2 k-1-\dl}\sigma^{\al+1}(\p^{k}_tw_x)^{2}\rt] d x\\
\ls&  (1+t)^{-1^+}\sum^{k-1}_{\ell=0}\mathcal{E}_{\ell}\\
 &+ \sum^{k-1}_{\ell=0}\int\lt[ (1+t)^{2\ell+1-\dl}\sigma^\al (\p^{\ell+1}_tw)^{2} +(1+t)^{2\ell-1-\dl}\sigma^{\al+1}(\p^{\ell}_tw_x)^{2}\rt] d x.
\eali
\ee
For the term $\sigma^{\al+1}(1+t)^{1+\la-\dl}J\p^k_t w_x$, the same estimate as \eqref{ej2} implies that
\bes
\bali
&\lt|(1+t)^{1+\la-\dl}\int \sigma^{\al+1}J\p^k_t w_xdx\rt|\\
&\ls \nu(1+t)^{-\dl}\int \sigma^{\al+1}  \lt(\partial_{t}^{k} w_{x}\rt)^2 d x\\
& +(1+t)^{-2k+1}\sum_{\ell=1}^{k-2}(1+t)^{2\ell-\dl}\int \sigma^{\al+1}\left(\partial_{t}^\ell w_{x}\right)^2dx\\
 &+\e^2_0(1+t)^{-2k-\frac{2\beta}{\gamma+1}}\sum^{k-1}_{\ell=0}\mathcal{E}_{\ell}.
\eali
\ees
From this, we have
\be\label{ehweighted9}
\bali
&(1+t)^{2k-1}\int\lt[{\mathfrak{E}}_k(t)+\sigma^{\al+1} (1+t)^{1+\la-\dl}J \partial_{t}^{k} w_{x}\rt]dx\\
&\ls \int \lt\{\lt[(1+t)^{2k+\la-\dl} \sigma^{\al} (\p^{k+1}_tw)^{2}+(1+t)^{2k-1-\dl}\sigma^{\al+1}(\p^k_tw_x)^2\rt]\rt.\\
&+\lt. (1+t)^{2k-1-\dl}\sigma^{\al} (\p^{k}_tw)^{2}\rt\}dx\\
&+\sum_{\ell=1}^{k-2}(1+t)^{2\ell-1-\dl}\int \sigma^{\al+1}\left(\partial_{t}^{\ell} w_{x}\right)^2dx\\
&+\e_0(1+t)^{-1-\frac{2\beta}{\gamma+1}}\sum^{k-1}_{\ell=0}\mathcal{E}_{\ell},
\eali
\ee
and
\be\label{ehweighted9r}
\bali
&(1+t)^{2k}\int\lt[{\mathfrak{E}}_k(t)+\sigma^{\al+1} (1+t)^{1+\la-\dl}J \partial_{t}^{k} w_{x}\rt]dx\\
&\thickapprox \mathcal{E}_{k}-\sum^{k-1}_{\ell=0}\mathcal{E}_{\ell}.
\eali
\ee
Inserting \eqref{ehweighted9} into \eqref{ehweighted10} and by choosing sufficiently large $N$, we can get

\bes
\bali
&\f{d}{dt}\lt\{(1+t)^{2k}\int\lt[{\mathfrak{E}}_k(t)+\sigma^{\al+1} (1+t)^{1+\la-\dl}J \partial_{t}^{k} w_{x}\rt]dx\rt\}\\
&+\int\lt[(1+t)^{2k+1-\dl} \sigma^{\al} (\p^{k+1}_tw)^{2}+(1+t)^{2k-1-\dl}\sigma^{\al+1}(\p^k_tw_x)^2\rt]\\
\ls&  (1+t)^{-1^+}\sum^{k-1}_{\ell=0}\mathcal{E}_{\ell}(t)\\
&+ \sum^{k-1}_{\ell=0}\int\lt[ (1+t)^{2\ell+1-\dl}\sigma^\al (\p^{\ell+1}_tw)^{2} +(1+t)^{2\ell-1-\dl}\sigma^{\al+1}(\p^{\ell}_tw_x)^{2}\rt] d x\\
\ls&  (1+t)^{-1^+}\sum^{k-1}_{\ell=0}\mathcal{E}_{\ell}(0)\\
&+ \sum^{k-1}_{\ell=0}\int\lt[ (1+t)^{2\ell+1-\dl}\sigma^\al (\p^{\ell+1}_tw)^{2} +(1+t)^{2\ell-1-\dl}\sigma^{\al+1}(\p^{\ell}_tw_x)^{2}\rt] d x.
\eali
\ees
Integrating the above inequality from $0$ to $t$ and remembering \eqref{ehweighted9r} and \eqref{ehenergy2}, we can get
\bes
\bali
&\mathcal{E}_{k}+ \int^t_0\lt[\int(1+\tau)^{2k+1-\dl} \sigma^\al (\p^{k+1}_tw)^{2} +(1+\tau)^{2k-1-\dl}\sigma^{\al+1}(\p^{k}_tw_x)^{2}\rt] d xd\tau\\
\ls&(1+t)^{2k}\int\lt[{\mathfrak{E}}_k(t)+\sigma^{\al+1} (1+t)^{1+\la-\dl}J \partial_{t}^{k} w_{x}\rt]dx+\sum^{k}_{\ell=0}\mathcal{E}_{\ell}(0)\\
&+ \int^t_0\lt[\int(1+\tau)^{2k+1-\dl} \sigma^\al (\p^{k+1}_tw)^{2} +(1+\tau)^{2k-1-\dl}\sigma^{\al+1}(\p^{k}_tw_x)^{2}\rt] d xd\tau\\
\ls&  \sum^{k}_{\ell=0}\mathcal{E}_{\ell}(0)+ \sum_{\ell=0}^{k-1}\int^t_0(1+\tau)^{2\ell-1-\dl}\int \sigma^{\al+1}\left(\partial_{t}^{\ell} w_{x}\right)^2dxd\tau\\
& +\sum_{\ell=0}^{k-1}\int^t_0(1+\tau)^{2\ell+1-\dl}\int \sigma^{\al}\left(\partial_{t}^{\ell+1} w\right)^2dxd\tau\\
\ls&  \sum^{k}_{\ell=0}\mathcal{E}_{\ell}(0).
\eali
\ees
This finishes the proof of Lemma \ref{ehweightede}.

\end{proof}

Then Propositions \ref{pelliptic} and Proposition \ref{proweightedenergy} together imply \eqref{eweight1}, which proves Theorem \ref{thmain} by continuation argument.
\section{Proof of Theorem \ref{thoriginal}}

\begin{proof} In this section, we prove Theorem \ref{thoriginal}. First, it follows from \eqref{emass}, \eqref{ecorf1}, and \eqref{eerro} that for $(x, t) \in \mathcal{I} \times[0, \infty)$
\[
\rho(\eta(x, t), t)-\bar{\rho}(\bar{\eta}(x, t), t)=\frac{\bar{\rho}_{0}(x)}{\eta_{x}(x, t)}-\frac{\bar{\rho}_{0}(x)}{\bar{\eta}_{x}(x, t)}=
-\bar{\rho}_{0}(x) \frac{w_{x}(x, t)+h(t)}{\left(\tilde{\eta}_{x}+w_{x}\right) \bar{\eta}_{x}(x, t)} ,
\]
and
\[
u(\eta(x, t), t)-\bar{u}(\bar{\eta}(x, t), t)=w_{t}(x, t)+x h_{t}(t) .
\]
Hence, by virtue of \eqref{elinfinity}, \eqref{ecor3}, \eqref{ecor4} and the boundedness of $h$, we have, for $(x, t) \in \mathcal{I} \times[0, \infty)$,
\bes \begin{aligned}
&|\rho(\eta(x, t), t)-\bar{\rho}(\bar{\eta}(x, t), t)| \\
\lesssim &\left(A-B x^{2}\right)^{\frac{1}{\gamma-1}}(1+t)^{-\frac{2(\la+1)}{\gamma+1}}\left((1+t)^{\f{\dl}{2}{\bf1}_{\la<1}}\sqrt{\mathcal{E}(0)}+1\right)\\
\ls&  \left(A-B x^{2}\right)^{\frac{1}{\gamma-1}}(1+t)^{-\frac{2(\la+1)}{\gamma+1}+\f{\dl}{2}{\bf1}_{\la<1}}.
\end{aligned}
\ees and
\bes
\bali
&|u(\eta(x, t), t)-\bar{u}(\bar{\eta}(x, t), t)|\\
 \lesssim&(1+t)^{-1+\f{\dl}{2}{\bf1}_{\la<1}} \sqrt{\mathcal{E}(0)}+(1+t)^{\f{\la-\dl}{\g+1}}\\
\ls& (1+t)^{\f{\la-\dl}{\g+1}}.
\eali
\ees
 Then \eqref{eod} and \eqref{eov} follow. It follows from \eqref{ean1}, \eqref{ecorf1} and \eqref{eerro} that
\bes
\begin{aligned}
x_{\pm}(t)=\eta\left(\bar{x}_{+}(0), t\right) &=(\tilde{\eta}+w)\left(\bar{x}_{\pm}(0), t\right) \\
&=(\bar{\eta}+x h+w)\left(\bar{x}_{+}(0), t\right) \\
&=\pm\sqrt{A B^{-1}}\left((1+t)^{\frac{\la+1}{\gamma+1}}+h(t)\right)+w(\sqrt{A B^{-1}}, t).
\end{aligned}
\ees
 Again using the boundedness of $h$ and \eqref{elinfinity}, we have
\bes
\bali
\pm\sqrt{A B^{-1}}(1+t)^{\frac{\la+1}{\gamma+1}}&-C\lt((1+t)^{\f{\dl}{2}{\bf1}_{\la<1}}\s{\mathcal{E}(0)}+1\rt)\\
                               \ls& x_{\pm}(t)\ls\\
\pm\sqrt{A B^{-1}}(1+t)^{\frac{\la+1}{\gamma+1}}&+C\lt((1+t)^{\f{\dl}{2}{\bf1}_{\la<1}}\s{\mathcal{E}(0)}+1\rt),
\eali
\ees
which implies \eqref{eob}. For $k=1,2,3$
\bes
\f{d^k x_{\pm}(t)}{dt^k}=\p^k_t\t{\eta}(\pm\s{AB^{-1}},t)+\p^k_t{w}(\pm\s{AB^{-1}},t). \ees
So using \eqref{ecor3}, \eqref{ecor4} and \eqref{elinfinity}, we get \eqref{eobd}.
\end{proof}

\begin{appendix}
\section{Estimates for ${\t{\eta}_x}$ and ${h}$}

In this appendix, we prove \eqref{ecor3} and \eqref{ecor4}. The idea of proof follows the line with that in Appendix of \cite{LZ:2016CPAM}. We may write \eqref{ecorf} as the following system:
\bes
\left\{
\bali
&h_{t}=z, \\
&z_{t}=-\f{\mu}{(1+t)^\la}z-\f{\mu(\la+1)}{\g+1}\left[\bar{\eta}_{x}^{-\gamma}-\left(\bar{\eta}_{x}+h\right)^{-\gamma}\right]-\bar{\eta}_{x t t}, \\
&(h, z)|_{t=0}=(0,0).
\eali
\right.
\ees

Recalling that $\bar{\eta}_{x}(t)=(1+t)^{\f{\la+1}{\g+1}},$ we have $\bar{\eta}_{x t t}<0 .$ A simple phase plane analysis shows that there exist $0<t_{0}<t_{1}<t_{2}$ such that, starting from $(h, z)=$
(0,0) at $t=0, h$ and $z$ increase in the interval $\left[0, t_{0}\right]$ and $z$ reaches its positive maximum at $t_{0} ;$ in the interval $\left[t_{0}, t_{1}\right], h$ keeps increasing and reaches its maximum at $t_{1}, z$ decreases from its positive maximum to $0 ;$ in the interval $\left[t_{1}, t_{2}\right],$ both $h$ and
$z$ decrease, and $z$ reaches its negative minimum at $t_{2} ;$ in the interval $\left[t_{2}, \infty\right), h$ decreases and $z$ increases, and $(h, z) \rightarrow(0,0)$ as $t \rightarrow \infty .$ This can be summarized as follows:
\[
\begin{array}{ll}
z(t) \uparrow_{0}, & h(t) \uparrow_{0}, \quad t \in\left[0, t_{0}\right] \\
z(t) \downarrow_{0}, & h(t) \uparrow, \quad t \in\left[t_{0}, t_{1}\right] \\
z(t) \downarrow^{0}, & h(t) \downarrow, \quad t \in\left[t_{1}, t_{2}\right] \\
z(t) \uparrow^{0}, & h(t) \downarrow_{0}, \quad t \in\left[t_{2}, \infty\right).
\end{array}
\]
It follows from the above analysis that there exists a constant $C=C(\la,\mu,\gamma, M)$ such that
\[
0 \leq h(t) \leq C \quad \text { for } t \geq 0 .
\]
In view of \eqref{ecorf1}, we then see that for some constant $K>0$
\[
(1+t)^{ \f{\la+1}{\g+1}} \leq \tilde{\eta}_{x} \leq K(1+t)^{ \f{\la+1}{\g+1}}.
\]
To derive the decay property, we may rewrite \eqref{ecorf2} as
\be\label{eapp2}
\lt\{
\bali
&\tilde{\eta}_{x t t}+\f{\mu}{(1+t)^\la}\tilde{\eta}_{x t}-\f{\mu(\la+1)}{\g+1}\tilde{\eta}_{x}^{-\gamma}=0,\\
&\tilde{\eta}_{x}|_{t=0}=1, \quad \tilde{\eta}_{x t}|_{t=0}=\f{\la+1}{\g+1}.
\eali
\rt.
\ee
Next we give the proof of \eqref{ecor3} and \eqref{ecor4} separately. The general ideas are the same.

\noindent{\bf Case 1: $\boldsymbol{\la<1}$:}

Then, we have by solving \eqref{eapp2} that
\be\label{eapp3}
\bali
\tilde{\eta}_{x t}(t)=&\f{\la+1}{\g+1}e^{-\f{\mu}{1-\la}(1+t)^{1-\la}}\\
 &+\frac{(\la+1)\mu}{\gamma+1} \int_{0}^{t} e^{\f{\mu}{1-\la}\lt[(1+s)^{1-\la}-(1+t)^{1-\la}\rt]}\tilde{\eta}_{x}^{-\gamma}(s) d s\\
  \geq& 0.
\eali
\ee
Next, we use the induction to prove $\eqref{ecor3}$. First, it follows from
\eqref{eapp3} that
\bes
\bali
\tilde{\eta}_{x t}(t) \les &e^{-\f{\mu}{1-\la}(1+t)^{1-\la}}+\int_{0}^{t} e^{\f{\mu}{1-\la}\lt[(1+s)^{1-\la}-(1+t)^{1-\la}\rt]}(1+s)^{-\f{(\la+1)\g}{\g+1}} d s\\
             \ls& e^{-\f{\mu}{1-\la}(1+t)^{1-\la}}+\lt\{\int_{0}^{t/2}+\int_{t/2}^{t}\rt\}e^{\f{\mu}{1-\la}\lt[(1+s)^{1-\la}-(1+t)^{1-\la}\rt]}(1+s)^{-\f{(\la+1)\g}{\g+1}} d s\\
             \ls& e^{-\f{\mu}{1-\la}(1+t)^{1-\la}}+e^{-c(1+t)^{1-\la}}\int_{0}^{t/2}(1+s)^{-\f{(\la+1)\g}{\g+1}} ds\\
             &+(1+t)^{-\f{(\la+1)\g}{\g+1}}\int_{t/2}^{t}e^{\f{\mu}{1-\la}\lt[(1+s)^{1-\la}-(1+t)^{1-\la}\rt]} d s\\
                       \les& e^{-\f{\mu}{1-\la}(1+t)^{1-\la}}+(1+t)^{-\f{(\la+1)\g}{\g+1}+\la} \\
                       \les&  (1+t)^{\f{\la+1}{\g+1}-1}.
 \eali
\ees
This proves $\eqref{ecor3}$ for $k=1 .$ For $k\geq 2$, we make the induction hypothesis that
\eqref{ecor3} holds for all $\ell=1,2, \ldots, k-1,$ that is
\be\label{eapp4}
\left|\frac{d^{\ell} \tilde{\eta}_{x}(t)}{d t^{\ell}}\right| \leq C_k (1+t)^{\frac{\la+1}{\gamma+1}-\ell}, \quad \ell=1,2, \ldots, k-1.
\ee
It suffices to prove \eqref{eapp4} holds for $\ell=k .$ We derive from \eqref{eapp2} that
\bes
\bali
&\frac{d^{k+1} \tilde{\eta}_{x}}{d t^{k+1}}(t)+\f{\mu}{(1+t)^\la}\frac{d^{k} \tilde{\eta}_{x}}{d t^{k}}(t)\\
=&\frac{(\la+1)\mu}{\gamma+1} \frac{d^{k-1} \tilde{\eta}_{x}^{-\gamma}}{d t^{k-1}}(t)-\mu\sum^{k-1}_{\ell=1}C^\ell_{k-1}\f{d^\ell (1+t)^{-\la}}{dt^\ell}\f{d^{k-\ell} \tilde{\eta}_{x}}{dt^{k-\ell}}.
\eali
\ees
Solving this ODE gives that
\be\label{eapp5}
\bali
&\frac{d^{k} \tilde{\eta}_{x}}{d t^{k}}(t)\\
=&e^{-\f{\mu}{1-\la}(1+t)^{1-\la}}\frac{d^{k} \tilde{\eta}_{x}}{d t^{k}}(0)\\
                              &+\frac{(\la+1)\mu}{\gamma+1}\int^t_0  e^{\f{\mu}{1-\la}\lt[(1+s)^{1-\la}-(1+t)^{1-\la}\rt]}\frac{d^{k-1} \tilde{\eta}_{x}^{-\gamma}}{d t^{k-1}}(s)ds\\
                               &-\mu\int^t_0  e^{\f{\mu}{1-\la}\lt[(1+s)^{1-\la}-(1+t)^{1-\la}\rt]}\sum^{k-1}_{\ell=1}C^\ell_{k-1}\f{d^\ell (1+s)^{-\la}}{dt^\ell}\f{d^{k-\ell} \tilde{\eta}_{x}}{dt^{k-\ell}}(s)ds
\eali
\ee
where $\frac{d^{k} \tilde{\eta}_{x}}{d t^{k}}(0)$ can be determined by the equation inductively. We need to bound the second and the third term on the righthand of \eqref{eapp5}.

Using \eqref{eapp4}, we can deduce by induction that
\be\label{eapp6}
\left|\frac{d^{\ell} \tilde{\eta}^{-\g}_{x}(t)}{d t^{\ell}}\right| \leq C_k (1+t)^{-\frac{(\la+1)\g}{\gamma+1}-\ell}, \quad \ell=1,2, \ldots, k-1.
\ee
Substituting \eqref{eapp4} and \eqref{eapp6} into \eqref{eapp5}, we can get

\bes
\bali
&\frac{d^{k} \tilde{\eta}_{x}}{d t^{k}}(t)\\
\ls&e^{-\f{\mu}{1-\la}(1+t)^{1-\la}}\\
&+\int^t_0  e^{\f{\mu}{1-\la}\lt[(1+s)^{1-\la}-(1+t)^{1-\la}\rt]}(1+s)^{-\frac{(\la+1)\g}{\gamma+1}-(k-1)}ds\\
&+\sum^{k-1}_{\ell=1}\int^t_0  e^{\f{\mu}{1-\la}\lt[(1+s)^{1-\la}-(1+t)^{1-\la}\rt] }(1+s)^{-\la-\ell}(1+s)^{-\frac{(\la+1)\g}{\gamma+1}-(k-\ell)}(s)ds\\
\ls&e^{-\f{\mu}{1-\la}(1+t)^{1-\la}}\\
&+\int^t_0  e^{\f{\mu}{1-\la}\lt[(1+s)^{1-\la}-(1+t)^{1-\la}\rt]}(1+s)^{-\frac{(\la+1)\g}{\gamma+1}-(k-1)}ds\\
\ls&e^{-\f{\mu}{1-\la}(1+t)^{1-\la}}+(1+t)^{-\frac{(\la+1)\g}{\gamma+1}-(k-1)+\la}\\
\ls& (1+t)^{\frac{(\la+1)}{\gamma+1}-k}.
\eali
\ees
 This finishes the proof of \eqref{ecor3}.

\noindent{\bf Case 2: $\boldsymbol{\la=1}$:}

We have by solving \eqref{eapp2} that
\be\label{eapp3r}
\tilde{\eta}_{x t}(t)=\frac{2}{\gamma+1}(1+t)^{-\mu}+\frac{2\mu}{\gamma+1} (1+t)^{-\mu}\int_{0}^{t} (1+s)^\mu\tilde{\eta}_{x}^{-\gamma}(s) d s \geq 0.
\ee
The same as case $0<\la<1$, we use the induction to prove $\eqref{ecor4}$ for $k<\mu+\f{2}{\g+1}$. First, it follows from
\eqref{eapp3r} that
\bes
\bali
\tilde{\eta}_{x t}(t) &\les (1+t)^{-\mu}+ (1+t)^{-\mu }\int_{0}^{t} (1+s)^\mu(1+s)^{-\f{2\g}{\g+1}} d s\\
                       &\les (1+t)^{-\mu}+ (1+t)^{-\mu } (1+t)^{\mu-\f{2\g}{\g+1}+1} \\
                       & =(1+t)^{-\mu}+ (1+t)^{\f{2}{\g+1}-1}\\
                       &\les  (1+t)^{\f{2}{\g+1}-1}.
\eali
\ees
This proves $\eqref{ecor4}$ for $k=1 .$ For $2 \leq k <\mu+\f{2}{\g+1}$, we make the induction hypothesis that
\eqref{ecor4} holds for all $\ell=1,2, \ldots, k-1,$ that is
\be\label{eapp4r}
\left|\frac{d^{\ell} \tilde{\eta}_{x}(t)}{d t^{\ell}}\right| \leq C_k (1+t)^{\frac{2}{\gamma+1}-\ell}, \quad \ell=1,2, \ldots, k-1.
\ee
It suffices to prove \eqref{eapp4r} holds for $\ell=k .$ We derive from \eqref{eapp2} that
\bes
\bali
&\frac{d^{k+1} \tilde{\eta}_{x}}{d t^{k+1}}(t)+\f{\mu}{1+t}\frac{d^{k} \tilde{\eta}_{x}}{d t^{k}}(t)\\
=&\frac{2\mu}{\gamma+1} \frac{d^{k-1} \tilde{\eta}_{x}^{-\gamma}}{d t^{k-1}}(t)-\mu\sum^{k-1}_{\ell=1}C^\ell_{k-1}\f{d^\ell (1+t)^{-1}}{ds^\ell}\f{d^{k-\ell} \tilde{\eta}_{x}}{ds^{k-\ell}}.
\eali
\ees
So that
\be\label{eapp5r}
\bali
&\frac{d^{k} \tilde{\eta}_{x}}{d t^{k}}(t)\\
=&(1+t)^{-\mu}\frac{d^{k} \tilde{\eta}_{x}}{d t^{k}}(0)\\
                              &+\frac{2\mu}{\gamma+1}(1+t)^{-\mu}\int^t_0 (1+s)^{\mu}\frac{d^{k-1} \tilde{\eta}_{x}^{-\gamma}}{d t^{k-1}}(s)ds\\
                               &-\mu(1+t)^{-\mu}\int^t_0 \sum^{k-1}_{\ell=1}(1+s)^{\mu}C^\ell_{k-1}\f{d^\ell (1+s)^{-1}}{dt^\ell}\f{d^{k-\ell} \tilde{\eta}_{x}}{dt^{k-\ell}}(s)ds. \eali
\ee
We need to bound the second and the third terms on the righthand of \eqref{eapp5r}.

Using \eqref{eapp4r}, we can deduce by induction that
\be\label{eapp6r}
\left|\frac{d^{\ell} \tilde{\eta}^{-\g}_{x}(t)}{d t^{\ell}}\right| \leq C_k (1+t)^{-\frac{2\g}{\gamma+1}-\ell}, \quad \ell=1,2, \ldots, k-1.
\ee
Substituting \eqref{eapp4r} and \eqref{eapp6r} into \eqref{eapp5r}, we can get
\be\label{eapp7}
\bali
\frac{d^{k} \tilde{\eta}_{x}}{d t^{k}}(t)\les&(1+t)^{-\mu}+(1+t)^{-\mu}\int^t_0 (1+s)^{\mu}(1+s)^{-\f{2\g}{\g+1}-(k-1)}ds\\
                               &+(1+t)^{-\mu}\sum^{k-1}_{\ell=1}\int^t_0 (1+s)^{\mu}(1+s)^{-1-\ell}(1+s)^{\f{2}{\g+1}-(k-\ell)}ds  \\
                               \les&(1+t)^{-\mu}+(1+t)^{-\mu}\int^t_0(1+s)^{\mu+\f{2}{\g+1}-k-1}ds\\
                               \les&(1+t)^{-\mu}+(1+t)^{-\mu}(1+t)^{\mu+\f{2}{\g+1}-k}\\
                               \les&(1+t)^{\f{2}{\g+1}-k}.
\eali
\ee
This finishes the proof of \eqref{eapp4r} for $k<\mu+\f{2}{\g+1}$. If $k=\mu+\f{2}{\g+1}$, from the third line of \eqref{eapp7}, we have
\bes
\bali
\lt|\frac{d^{k} \tilde{\eta}_{x}}{d t^{k}}(t)\rt|
                               \les&(1+t)^{-\mu}+(1+t)^{-\mu}\int^t_0(1+s)^{-1}ds\\
                               \les&(1+t)^{-\mu}\ln(1+t).
\eali
\ees
When $k>\mu+\f{2}{\g+1}$, it is a routine work to prove that
\bes
\lt|\frac{d^{k} \tilde{\eta}_{x}}{d t^{k}}(t)\rt|\leq (1+t)^{-\mu}\ln(1+t) .
\ees
by again using induction. Since in this case, the terms
\bes
\int^t_0 (1+s)^{\mu}\frac{d^{k-1} \tilde{\eta}_{x}^{-\gamma}}{d t^{k-1}}(s)ds,\q \int^t_0 (1+s)^{\mu}\f{d^\ell (1+s)^{-1}}{dt^\ell}\f{d^{k-\ell} \tilde{\eta}_{x}}{dt^{k-\ell}}(s)ds
\ees
 in \eqref{eapp5r}, actually are bounded by $\ln(1+t)$. This finishes the proof of \eqref{ecor4}. \ef

\end{appendix}

\section*{Acknowledgments}  The author would like to thank Prof. Huicheng Yin and Dr. Fei Hou in Nanjing Normal University for helpful conversation.

\bibliographystyle{plain}



\def\cprime{$'$}


\end{document}